\documentclass[review,onefignum,onetabnum]{siamart171218}

\usepackage{enumerate}
\usepackage{amsmath}
\usepackage{mathrsfs}
\usepackage{xcolor}	
\usepackage{amssymb}
\usepackage{bm}
\usepackage{graphicx}
\usepackage{subcaption}
\usepackage{adjustbox}
\usepackage{soul}
\usepackage{listings}
\usepackage{array}  

\usepackage{amssymb,latexsym}
\usepackage{multirow}

\usepackage{color,epstopdf}
\usepackage{enumitem}
\usepackage{setspace}
\usepackage{url}
\usepackage{scalerel}
\usepackage[ruled,vlined,algo2e,linesnumbered]{algorithm2e}

\usepackage{footnote}
\makesavenoteenv{algorithm}
\usepackage{mdframed}
\newmdenv[linecolor=green!50!black, fontcolor=green!50!black, backgroundcolor=green!20, linewidth=2pt, roundcorner=10pt]{gnote}

\definecolor{amber}{rgb}{1.0, 0.49, 0.0}
\definecolor{darkpastelgreen}{rgb}{0.01, 0.75, 0.24}
\definecolor{darkred}{rgb}{0.64, 0.0, 0.0}
\definecolor{amberhl}{rgb}{1.0, 0.75, 0.0}

\newmdenv[linecolor=blue!50!black, fontcolor=blue!50!black, backgroundcolor=blue!20, linewidth=2pt, roundcorner=10pt]{anote}

\newcolumntype{L}{>{$}l<{$}}

\headers{NPASA - Motivation and Global Convergence}{J. Diffenderfer, W. W. Hager}

\title{NPASA: An algorithm for nonlinear programming - Motivation and Global Convergence\thanks{Submitted to the editors \textcolor{red}{ADD DATE}.\funding{The authors gratefully acknowledge support by the National Science Foundation under Grant 1819002, and by the Office of Naval Research under Grants N00014-15-1-2048 and N00014-18-1-2100. This work was performed under the auspices of the U.S. Department of Energy by Lawrence Livermore National Laboratory  under  Contract  DE-AC52-07NA27344, LLNL-JRNL-824568-DRAFT.}}}
\author{James Diffenderfer\thanks{Lawrence Livermore National Laboratory, Livermore, CA (\email{diffenderfer2@llnl.gov})}
  \and William W. Hager\thanks{The University of Florida,
    Gainesville, FL (\email{hager@ufl.edu}, \url{http://people.clas.ufl.edu/hager/})}}

\ifpdf
\hypersetup{
  pdftitle={NPASA: An algorithm for nonlinear programming - Motivation and Global Convergence},
  pdfauthor={J. Diffenderfer, W. W. Hager}
}
\fi

\RequirePackage[normalem]{ulem} %DIF PREAMBLE
\RequirePackage{color}\definecolor{RED}{rgb}{1,0,0}\definecolor{BLUE}{rgb}{0,0,1} %DIF PREAMBLE
\begin{document}

\maketitle

% REQUIRED
\begin{abstract}
In this paper, we present a two phase method for solving nonlinear programming problems called Nonlinear Polyhedral Active Set Algorithm (NPASA) that has global and local convergence guarantees under reasonable assumptions. The first phase consists of an augmented Lagrangian method to ensure global convergence while the second phase is designed to promote fast local convergence by performing a balanced reduction of two error estimators for nonlinear programs. After presenting error estimators for nonlinear programs and our algorithm NPASA, we establish global convergence properties for NPASA. Local quadratic convergence of NPASA is established in a companion paper \cite{diffenderfer2020local}.
\end{abstract}

% REQUIRED

% REQUIRED

% Body
\section {Introduction}
\label{section:introduction}
Some of the most successful approaches for solving nonlinear programs make use of augmented Lagrangian, sequential quadratic programming (SQP), or interior-point methods \cite{NandW}. Early progress towards the implementation of a general purpose algorithm for solving nonlinear programs was made for nonlinear programs with linear constraints in the late 1970s \cite{Murtagh1978} in the form of the software MINOS. Some theoretical support for using the method of multipliers to solve nonlinear programs was established in \cite{Polak1980} and the first significant general purpose algorithm for large-scale nonlinear programs with nonlinear constraints was published in 1991 \cite{Conn1991}. This algorithm, implemented under the name LANCELOT \cite{Conn1996}, made use of augmented Lagrangian techniques and its success encouraged the continued study \cite{Birgin2005, Conn2010, Izmailov2012, Lewis2002} and use \cite{Andreani2008} of augmented Lagrangian techniques for solving nonlinear programs. Some notable implementations of a SQP approach for solving problem nonlinear programs are DNOPT \cite{GilSWdnopt}, FILTERSQP \cite{Fletcher1998}, KNITRO \cite{Byrd2006}, and SNOPT \cite{Gill2005, snopt77}. As a note on the usefulness of augmented Lagrangian methods, the SQP implementations SNOPT and DNOPT make use of an augmented Lagrangian technique to ensure global convergence. Originally introduced in the 1960s \cite{Wilson1963}, early refinements to the theory of SQP techniques in the 1970s \cite{Han1976, Palomares1976, Powell1978} provided better performance over augmented Lagrangian techniques by attaining superlinear convergence rates under certain assumptions. Another approach that has been successful in practice is interior-point methods. Originally referred to as barrier methods, this approach was first introduced in the 1950s \cite{Frisch1955}. A recent implementation of interior point techniques, called IPOPT \cite{Wachter2006}, achieves global convergence under few assumptions \cite{Wachter2005G} and it also has a local superlinear convergence rate under some assumptions after it was shown to not experience the Maratos effect \cite{Wachter2005L}. %Additionally, IPOPT has been successful in practice. 
Benchmarking results comparing several of the listed implementations can be found at \cite{NLPBenchmark}. %seems that KNITRO is best with IPOPT second

In this paper, we present a method for solving a general nonlinearly constrained optimization problem based on the recently developed \cite{HagerActive} polyhedral active set algorithm (PASA), an efficient algorithm for optimizing a nonlinear function over a polyhedron. Let us consider the nonlinear programming problem
\begin{align}
\begin{array}{cc}
\displaystyle \min_{\bm{x} \in \mathbb{R}^{n}} & f(\bm{x}) \\
\text{s.t.} & \bm{h}(\bm{x}) = \bm{0}, \ \bm{r} (\bm{x}) \leq \bm{0}
\end{array} \label{prob:main-nlp}
\end{align}
where $f : \mathbb{R}^n \to \mathbb{R}$ and $\bm{h} : \mathbb{R}^n \to \mathbb{R}^{\ell}$ are nonlinear functions and $\bm{r} : \mathbb{R}^n \to \mathbb{R}^m$ is a linear function defined by
\begin{align}
\bm{r} (\bm{x}) := \bm{A} \bm{x} - \bm{b}, \label{def:r(x)}
\end{align}
where $\bm{A} \in \mathbb{R}^{m \times n}$ and $\bm{b} \in \mathbb{R}^m$. For simplicity, we also denote the polyhedral constraint set for problem (\ref{prob:main-nlp}) by $\Omega = \{ \bm{x} \in \mathbb{R}^n: \bm{r} (\bm{x}) \leq \bm{0} \}$. Any nonlinear programming problem can be written in this form. It is assumed that $f$ and $\bm{h}$ are at least continuously differentiable. The algorithm for solving the nonlinearly constrained optimization problem (\ref{prob:main-nlp}) is based on the following steps consisting of polyhedral constrained subproblems:
\begin{enumerate}
\item An augmented Lagrangian step subject to the polyhedral constraints. This step is used to ensure convergence to a stationary point even with a poor starting guess.
\item A Newton iteration applied to the nonlinear constraints. This step amounts to a projection onto a polyhedral set.
\item A quadratic program for estimating the multipliers.
\item The optimization of an augmented Lagrangian over the tangent space associated with the nonlinear constraints.
\end{enumerate} 
In this way, the general nonlinearly constrained optimization problem is reduced to the solution of polyhedral constrained subproblems. Under suitable assumptions, the solution of subproblems 2 -- 4 is locally quadratically convergent. If the starting guess is good enough, the algorithm simply loops over the final three subproblems until the convergence tolerance is satisfied. Hence, our algorithm for solving the nonlinearly constrained optimization problems uses PASA to solve a series of polyhedral constrained optimization problems and is denoted NPASA. This paper establishes a global convergence result for NPASA, while the companion paper \cite{diffenderfer2020local} establishes a local quadratic convergence result.

The remainder of the paper is organized as follows. In Section~\ref{sec:Error-Estimators}, we discuss the error estimators used by NPASA for solving problem (\ref{prob:main-nlp}). In Section~\ref{section:npasa}, we provide pseudocode for the NPASA algorithm together with a discussion providing some motivation behind each step in the algorithm. Sections~\ref{section:npasa-phase-one} and \ref{section:npasa-phase-two} focus on establishing convergence results for the various subproblems solved within NPASA. We state and prove the main global convergence result for NPASA, Theorem~\ref{thm:npasa-global-conv}, in Section~\ref{section:npasa-global-convergence}. We now specify some notation used throughout the paper.

\subsection{Notation} \label{subsec:notation}
We will write $\mathbb{R}_+$ to denote the set $\{ x \in \mathbb{R} : x \geq 0 \}$. Scalar variables are denoted by italicized characters, such as $x$, while vector variables are denoted by boldface characters, such as $\bm{x}$. For a vector $\bm{x}$ we write $x_j$ to denote the $j$th component of $\bm{x}$ and we write $\bm{x}^{\intercal}$ to denote the transpose of $\bm{x}$. Outer iterations of NPASA will be enumerated by boldface characters with subscript $k$, such as $\bm{x}_k$, and we will write $x_{kj}$ to denote the $j$th component of the iterate $\bm{x}_k$. Inner iterations of NPASA will be enumerated by boldface characters with subscript $i$. If $\mathcal{S}$ is a subset of indices of $\bm{x}$ then we write $\bm{x}_{\mathcal{S}}$ to denote the subvector of $\bm{x}$ with indices $\mathcal{S}$. Additionally, for a matrix $\bm{M}$ we write $\bm{M}_{\mathcal{S}}$ to denote the submatrix of $\bm{M}$ with row indices $\mathcal{S}$. The ball with center $\bm{x}$ and radius $r$ will be denoted by $\mathcal{B}_r (\bm{x})$. Real valued functions are denoted by italicized characters, such as $f(\cdot)$, while vector valued functions are denoted by boldface characters, such as $\bm{h}(\cdot)$. The gradient of a real valued function $f(\bm{x})$ is denoted by $\nabla f (\bm{x})$ and is a row vector. The Jacobian matrix of a vector valued function $\bm{h} : \mathbb{R}^n \to \mathbb{R}^{\ell}$ is denoted by $\nabla \bm{h} \in \mathbb{R}^{\ell \times n}$. Given a vector $\bm{x}$, we will write $\mathcal{A} (\bm{x})$ denote the set of active constraints at $\bm{x}$, that is $\mathcal{A} (\bm{x}) := \{ i \in \mathbb{N} : \bm{a}_i^{\intercal} \bm{x} = 0 \}$ where $\bm{a}_i^{\intercal}$ is the $i$th row of matrix $\bm{A}$. For an integer $j$, we will write $f \in \mathcal{C}^j$ to denote that the function $f$ is $j$ times continuously differentiable. We write $c$ to denote a generic nonnegative constant that takes on different values in different inequalities. Given an interval $[a,b] \subset \mathbb{R}$, we write $Proj_{[a,b]} (\bm{x})$ to denote the euclidean projection of each component of the vector $\bm{x}$ onto the interval $[a,b]$. That is, if $\bm{v} = Proj_{[a,b]} (\bm{x})$ then the $j$th component of $\bm{v}$ is given by
\begin{align}
v_j = \left\{
	\begin{array}{lll}
	a & : & x_j \leq a \\
	x_j & : & a < x_j < b \\
    b & : & x_j \geq b
	\end{array}
	\right. .
\end{align}

Additionally, in order to simplify the statement of several results throughout the paper, we provide abbreviations for assumptions that are used. Note that each assumption will be clearly referenced when required.
\begin{enumerate}[leftmargin=2\parindent,align=left,labelwidth=\parindent,labelsep=7pt]
%\item[(\textbf{LI})] Linear Independence: Given $\bm{x}$, the matrix {\footnotesize $\displaystyle \begin{bmatrix} \nabla \bm{h} (\bm{x}) \\ \bm{A} \end{bmatrix}$} is of full row rank.
\item[(\textbf{LICQ})] Linear Independence Constraint Qualification: Given $\bm{x}$, {\footnotesize $ \begin{bmatrix} \nabla \bm{h} (\bm{x}) \\ \bm{A}_{\mathcal{A}(\bm{x})} \end{bmatrix}$} has full row rank.
\item[(\textbf{SCS})] Strict complementary slackness holds at a minimizer of problem (\ref{prob:main-nlp}). That is, Definition~\ref{def:scs} holds at a minimizer of problem (\ref{prob:main-nlp}).
%\item[(\textbf{SOSC})] The second-order sufficient optimality conditions hold for a feasible point in problem (\ref{prob:main-nlp}). That is, the hypotheses of Theorem~\ref{thm:sosc} hold.
%\item[(\textbf{SSOSC})] The strong second-order sufficient optimality conditions hold for a feasible point in problem (\ref{prob:main-nlp}). That is, the hypotheses of Theorem~\ref{thm:ssosc} hold.
\end{enumerate}

\section{Error Estimators for Nonlinear Optimization} \label{sec:Error-Estimators}
In \cite{Hager2014}, the authors presented two error estimators for optimization problems of the form
\begin{align}
\begin{array}{cc}
\displaystyle \min_{\bm{x} \in \mathbb{R}^{n}} & f(\bm{x}) \\
\text{s.t.} & \bm{h}(\bm{x}) = \bm{0}, \ \bm{x} \geq \bm{0}.
\end{array}
\end{align}
As a natural generalization to the work %done by Hager and Mico-Umutesi 
in \cite{Hager2014}, we provide minor updates to these error estimators for use with problem (\ref{prob:main-nlp}). Both error estimators make use of the Lagrangian function $\mathcal{L}: \mathbb{R}^n \times \mathbb{R}^{\ell} \times \mathbb{R}^m \to \mathbb{R}$ for problem (\ref{prob:main-nlp}) defined by
\begin{align}
\mathcal{L} (\bm{x}, \bm{\lambda}, \bm{\mu}) = f(\bm{x}) + \bm{\lambda}^{\intercal} \bm{h}(\bm{x}) + \bm{\mu}^{\intercal} \bm{r}(\bm{x}).
\end{align}
The first error estimator is defined over the set $\mathcal{D}_0 := \{ (\bm{x}, \bm{\lambda}, \bm{\mu}) : \bm{x} \in \Omega, \bm{\lambda} \in \mathbb{R}^{\ell}, \bm{\mu} \geq \bm{0} \}$ and denoted $E_0 : \mathcal{D}_0 \to \mathbb{R}$ by 
\begin{align}
E_0 (\bm{x}, \bm{\lambda}, \bm{\mu}) = \sqrt{ \|\nabla_x \mathcal{L}(\bm{x}, \bm{\lambda}, \bm{\mu})\|^2 + \|\bm{h}(\bm{x})\|^2 - \bm{\mu}^{\intercal} \bm{r}(\bm{x})}. \label{def:E0}
\end{align}
%In addition to being a generalization of $E_0$ from \cite{Hager2014}, 
We note that $E_0$ here is a modified version of $E_0$ as it was originally defined in \cite{Hager2014}. In particular, if we were to generalize $E_0$ for problem (\ref{prob:main-nlp}) as defined in \cite{Hager2014} then the term $-\bm{\mu}^{\intercal} \bm{r}(\bm{x})$ would be squared in (\ref{def:E0}). This %minor 
modification to $E_0$ results in beneficial properties that can be exploited in the global and local convergence analysis of NPASA. Additionally, in Corollary~\ref{cor:E0-error-bound} we will show that $E_0$ as defined in (\ref{def:E0}) satisfies an error bound of the same form as $E_1$ regardless of whether or not strict complementary slackness holds which was previously required for $E_0$ as defined in \cite{Hager2014}. 
\iffalse
In section~\ref{subsec:ErrorAnalysis}, we will show that $E_0$ satisfies the error bound 
\begin{align}
\|\bm{x} - \bm{x}^*\| + \|\bm{\lambda} - \hat{\bm{\lambda}}\| + \|\bm{\mu} - \hat{\bm{\mu}}\| \leq c E_0 (\bm{x}, \bm{\lambda}, \bm{\mu}) \label{eq:E0-bound}
\end{align}
for all $(\bm{x}, \bm{\lambda}, \bm{\mu}) \in \mathcal{D}_0$ in a neighborhood of a KKT point $(\bm{x}^*, \bm{\lambda}^*, \bm{\mu}^*)$ and some constant $c$. Here we write $(\hat{\bm{\lambda}}, \hat{\bm{\mu}})$ to denote the projection of the multipliers $(\bm{\lambda}, \bm{\mu})$ onto the set of KKT multipliers at a stationary point $\bm{x}^*$. 
\fi

A second error estimator is considered that removes the restricted domain required by $E_0$. %, namely that $\bm{x} \in \Omega$ and $\bm{\mu} \geq 0$. 
This estimator makes use of the componentwise minimum function, denoted here by $\bm{\Phi} : \mathbb{R}^m \times \mathbb{R}^m \to \mathbb{R}^m$, so that the $i$th component is defined by
\begin{align}
\Phi_i \left(\bm{x}, \bm{y} \right) = \min\{ x_i, y_i\},
\end{align}
for $1 \leq i \leq m$. Then the second error estimator $E_1: \mathbb{R}^n \times \mathbb{R}^{\ell} \times \mathbb{R}^m \to \mathbb{R}$ is defined by 
\begin{align}
E_1 (\bm{x}, \bm{\lambda}, \bm{\mu}) = \sqrt{ \|\nabla_x \mathcal{L}(\bm{x}, \bm{\lambda}, \bm{\mu})\|^2 + \|\bm{h}(\bm{x})\|^2 + \|\bm{\Phi}\left(-\bm{r}(\bm{x}), \bm{\mu} \right)\|^2}. \label{def:E1}
\end{align}
The error bound property for $E_1$ is provided in Theorem~\ref{thm:E1-error-bound}.
\iffalse
In section~\ref{subsec:ErrorAnalysis}, we will show that $E_1$ also satisfies the error property given in (\ref{eq:E0-bound}) without the restriction to $\mathcal{D}_0$, that is,
\begin{align}
\|\bm{x} - \bm{x}^*\| + \|\bm{\lambda} - \hat{\bm{\lambda}}\| + \|\bm{\mu} - \hat{\bm{\mu}}\| \leq c E_1 (\bm{x}, \bm{\lambda}, \bm{\mu}) \label{eq:E1-bound}
\end{align}
for all $(\bm{x}, \bm{\lambda}, \bm{\mu})$ in a neighborhood of a KKT point $(\bm{x}^*, \bm{\lambda}^*, \bm{\mu}^*)$ and for some constant $c$.
\fi

In our discussion of NPASA, it will be useful to split the error estimators $E_0$ and $E_1$ into two parts. To this end, we define the \emph{multiplier} portion of the error estimators $E_0$ and $E_1$ by
\begin{align}
E_{m,0} (\bm{x}, \bm{\lambda}, \bm{\mu}) := \|\nabla_x \mathcal{L}(\bm{x}, \bm{\lambda}, \bm{\mu})\|^2 - \bm{\mu}^{\intercal} \bm{r}(\bm{x}) \label{def:Em0}
\end{align}
and
\begin{align}
E_{m,1} (\bm{x}, \bm{\lambda}, \bm{\mu}) := \|\nabla_x \mathcal{L}(\bm{x}, \bm{\lambda}, \bm{\mu})\|^2 + \|\bm{\Phi}\left(-\bm{r}(\bm{x}), \bm{\mu} \right)\|^2, \label{def:Em1}
\end{align}
respectively, and the \emph{constraint} portion of the error estimators by
\begin{align}
E_c (\bm{x}) := \| \bm{h} (\bm{x}) \|^2. \label{def:Ec}
\end{align}
Then from the definitions of $E_0$ and $E_1$ in (\ref{def:E0}) and (\ref{def:E1}) it follows that
\begin{align}
E_j (\bm{x}, \bm{\lambda}, \bm{\mu})^{2} = E_{m,j} (\bm{x}, \bm{\lambda}, \bm{\mu}) + E_c (\bm{x}), \label{eq:E=Em+Ec}
\end{align}
for $j \in \{0, 1\}$. 

As will be evident throughout the analysis in the following sections, there are practical benefits and drawbacks to each error estimator provided in this section. We briefly highlight some of the key differences between the error estimators to provide some insight on when each estimator is most useful. The first difference illustrated is that $E_1$ provides a tighter bound than $E_0$. This is established in the following Lemma. The proof is omitted as it is straightforward.
\begin{lemma} \label{lem:E0-E1-relationship}
Suppose $(\bm{x}, \bm{\lambda}, \bm{\mu}) \in \mathcal{D}_0 := \{ (\bm{x}, \bm{\lambda}, \bm{\mu}) : \bm{x} \in \Omega, \bm{\lambda} \in \mathbb{R}^{\ell}, \bm{\mu} \geq \bm{0} \}$. Then 
\begin{align}
E_{m,1} (\bm{x}, \bm{\lambda}, \bm{\mu}) \leq E_{m,0} (\bm{x}, \bm{\lambda}, \bm{\mu}). \label{result:E0-E1-relationship}
\end{align}
and
\begin{align}
E_1 (\bm{x}, \bm{\lambda}, \bm{\mu}) \leq E_0 (\bm{x}, \bm{\lambda}, \bm{\mu}). \label{result:E0-E1-relationship2}
\end{align}
\end{lemma}
%%%%%%%%%%%%%%%%%%%%%%%%%%%%%%%%%%%%%

We now provide the error bound properties for $E_0$ and $E_1$. We note that the theorem and proof only require slight modifications to the statement and proof of Theorem 3.1 in \cite{Hager2014}. %Following this proof we will establish that (\ref{eq:E0-bound}) holds. 
As such, we omit the proof. The error bound result for the modified $E_0$ then immediately follows as a corollary to Theorem~\ref{thm:E1-error-bound} when Lemma~\ref{lem:E0-E1-relationship} is applied.

\begin{theorem} \label{thm:E1-error-bound}
Suppose that $\bm{x}^*$ is a local minimizer of problem (\ref{prob:main-nlp}) and that $f, \bm{h} \in \mathcal{C}^2$ at $\bm{x}^*$. If there exists $\bm{\lambda}^*$ and $\bm{\mu}^* \geq \bm{0}$ such that the KKT conditions and the second-order sufficient optimality conditions hold satisfied at $(\bm{x}^*, \bm{\lambda}^*, \bm{\mu}^*)$ then there exists a neighborhood $\mathcal{N}$ of $\bm{x}^*$ and a constant $c$ such that
\begin{align}
\|\bm{x} - \bm{x}^*\| + \|\bm{\lambda} - \hat{\bm{\lambda}}\| + \|\bm{\mu} - \hat{\bm{\mu}}\| 
\leq c E_1 (\bm{x}, \bm{\lambda}, \bm{\mu}) \label{result:E1-error-bound}
\end{align}
for all $\bm{x} \in \mathcal{N}$ where $(\hat{\bm{\lambda}}, \hat{\bm{\mu}})$ denotes the projection of $(\bm{\lambda}, \bm{\mu})$ onto the set of KKT multipliers at $\bm{x}^*$. 
\end{theorem}

\begin{corollary} \label{cor:E0-error-bound}
Suppose that the hypotheses of Theorem~\ref{thm:E1-error-bound} are satisfied at $(\bm{x}^*, \bm{\lambda}^*, \bm{\mu}^*)$. Then there exists a neighborhood $\mathcal{N}$ of $\bm{x}^*$ and a constant $c$ such that
\begin{align}
\|\bm{x} - \bm{x}^*\| + \|\bm{\lambda} - \hat{\bm{\lambda}}\| + \|\bm{\mu} - \hat{\bm{\mu}}\| 
\leq c E_0 (\bm{x}, \bm{\lambda}, \bm{\mu}) \label{result:E0-error-bound}
\end{align}
for all $(\bm{x}, \bm{\lambda}, \bm{\mu}) \in \mathcal{D}_0$ with $\bm{x} \in \mathcal{N}$ where $(\hat{\bm{\lambda}}, \hat{\bm{\mu}})$ denotes the projection of $(\bm{\lambda}, \bm{\mu})$ onto the set of KKT multipliers at $\bm{x}^*$. 
\end{corollary}
Based on the error bounds in (\ref{result:E1-error-bound}) and (\ref{result:E0-error-bound}) it follows that $E_0$ and $E_1$ are useful for determining when to terminate an iterative method for locating a stationary point of problem (\ref{prob:main-nlp}). In particular, criterion for stopping an iterative method for solving problem (\ref{prob:main-nlp}) could be when either $E_0 (\bm{x}, \bm{\lambda}, \bm{\mu}) \leq \varepsilon$ or $E_1 (\bm{x}, \bm{\lambda}, \bm{\mu}) \leq \varepsilon$ for some small constant $\varepsilon$. 
%At such a point, the program could stop and return the point $(\bm{x}, \bm{\lambda}, \bm{\mu})$. 
A practical take-away from Lemma~\ref{lem:E0-E1-relationship} is that $E_1$ is a better choice than $E_0$ for stopping criterion in NPASA. 
%With these error bound results established we now update the active set algorithm from \cite{Hager2014} for solving (\ref{prob:E0-E1-minimization}) when $j = 1$.

We now consider a setting in which $E_0$ is more beneficial than $E_1$. Based on the bounds (\ref{result:E1-error-bound}) and (\ref{result:E0-error-bound}) established in Theorem~\ref{thm:E1-error-bound} and Corollary~\ref{cor:E0-error-bound}, it is of interest of minimize either $E_0$ or $E_1$ subject to the constraints $\bm{x} \in \Omega$ and $\bm{\mu} \geq \bm{0}$. In particular, for a fixed vector $\bm{\bar{x}} \in \Omega$ near a stationary point it is of interest to %minimize an error estimator over the dual variables $\bm{\lambda}$ and $\bm{\mu} \geq \bm{0}$ by solving the problem
solve the problem
\begin{align}
\begin{array}{cc}
\displaystyle \min_{\bm{\lambda}, \bm{\mu}} & E_j (\bm{\bar{x}}, \bm{\lambda}, \bm{\mu})^2 \\
\text{s.t.} & \bm{\mu} \geq \bm{0},
\end{array} \label{prob:E0-E1-minimization}
\end{align}
for some $j \in \{0, 1\}$. In such a setting, we claim that it may be easier to choose $j = 0$ instead of $j = 1$. To support this claim, suppose that $\mathcal{B}$ is a neighborhood of $\bm{\bar{x}}$ such that $f, \bm{h} \in \mathcal{C}^t (\mathcal{B})$, for some positive integer $t$. By definition of $E_0$, it follows immediately that $E_0 (\bm{x}, \bm{\lambda}, \bm{\mu})^2 \in \mathcal{C}^t (\mathcal{B})$. As such, differentiable methods can be used to solve (\ref{prob:E0-E1-minimization}) when $j = 0$. However, due to the nature of the componentwise minimum function in the definition of $E_1$ in (\ref{def:E1}), it is not necessarily true that $E_1 (\bm{x}, \bm{\lambda}, \bm{\mu})^2 \in \mathcal{C}^t (\mathcal{B})$. Hence, alternative techniques, such as the active set algorithm in \cite{Hager2014}, would be required to solve (\ref{prob:E0-E1-minimization}) when $j = 1$. Additionally, differentiability of the error estimator is useful in establishing desirable local convergence results in the companion paper \cite{diffenderfer2020local}. Ultimately, our method will leverage the strengths of each error estimator for solving problem (\ref{prob:main-nlp}). We conclude this comparison by providing conditions under which $E_1$ is differentiable in the following result. The proof is omitted as it is straightforward.
\begin{lemma} \label{lem:E1diff}
Let $(\bm{x}^*, \bm{\lambda}^*, \bm{\mu}^*)$ be a KKT point for problem (\ref{prob:main-nlp}) that satisfies assumption (\textbf{SCS}) and suppose that $f, \bm{h} \in \mathcal{C}^2$ in a neighborhood of $\bm{x}^*$. Then there exists a neighborhood $\mathcal{N}$ about $(\bm{x}^*, \bm{\lambda}^*, \bm{\mu}^*)$ such that $E_1 (\bm{x}, \bm{\lambda}, \bm{\mu}) \in \mathcal{C}^2$, for all $(\bm{x}, \bm{\lambda}, \bm{\mu}) \in \mathcal{N}$.
\end{lemma}
\section{NPASA: Algorithm for Nonlinear Programming Problems} \label{section:npasa}
In this section, we present details of the Nonlinear Polyhedral Active Set Algorithm (NPASA) designed to solve problem (\ref{prob:main-nlp}). %Below, we provide an outline which establishes a method for solving a nonlinear constrained optimization problem of the form (\ref{prob:main-nlp}). We note that a similar method for solving problem (\ref{prob:main-nlp}) was presented in \cite{Hager1993}, however, several updates designed to improve performance and robustness have been made some of which we briefly outline here. First, NPASA employs the updated error estimators $E_0$ and $E_1$ that satisfy the error bound properties established in Theorem~\ref{thm:E1-error-bound} and Corollary~\ref{cor:E0-error-bound}. These estimators are an improvement over the original error estimator used in \cite{Hager1993} that did not account for complimentarity slackness. Second, the technique used for minimizing the constraint error has been updated with the goal of improving robustness. Additionally, NPASA makes use of several recently developed methods, such as PPROJ \cite{Hager2016} and PASA \cite{HagerActive}, for solving subproblems to improve performance. In addition to using PASA, we propose some minor modifications to PASA in Section~\ref{section:npasa-phase-one} that provide benefits when solving one of the subproblems used in our approach.
Together with the pseudocode for NPASA, we provide a high-level discussion that serves to justify and motivate our approach. 
%This section is organized as follows. In Section~\ref{section:npasa-alg}, we present psuedocode for the Nonlinear Polyhedral Active Set Algorithm (NPASA) designed to solve problem (\ref{prob:main-nlp}) and a provide a high-level discussion with motivation for each component of the algorithm. Section~\ref{sec:mod-pasa-alg} contains the modified PASA algorithm for solving a problem arising in the NPASA algorithm. 
%Finally, Section~\ref{subsec:ASA} updates the previously developed active set algorithm, Figure 2 in \cite{Hager2014}, to account for the more general nonlinear program formulation in (\ref{prob:main-nlp}).

\subsection{NPASA Overview and Pseudocode} \label{section:npasa-alg}
NPASA is organized into two phases: %with the goal of performing a balanced reduction of the \emph{multiplier} and \emph{constraint} portions of the error estimator. 
Phase one contains an algorithm designed to ensure global convergence of NPASA, the Global Step (GS) algorithm, while phase two contains an algorithm designed to achieve fast local convergence by performing a balanced reduction of the constraint and multiplier error estimators, the Local Step (LS) algorithm. 
%We now introduce the GS and LS algorithms with their respective pseudocodes in Algorithm~\ref{alg:gs} and Algorithm~\ref{alg:cms}. %, respectively. 
NPASA improves upon a dual algorithm established by Hager \cite{Hager1993} through the use of enhanced error estimators, a cleverly designed phase two that can provide quadratic local convergence under reasonable assumptions, and by leveraging advances for solving polyhedral constrained optimization problems \cite{HagerActive, Hager2016}.
Note that we will write $(\bm{x}_k, \bm{\lambda}_k, \bm{\mu}_k)$ to denote the current primal-dual iterate of NPASA.

\begin{center}
\begin{algorithm}[!t]
%	\algsetup{linenosize=\notsotiny}
%	\notsotiny
	\caption{GS - Global Step Algorithm} \label{alg:gs} 
%	\begin{algorithmic}[1]
        % Initialize parameters
        {\bfseries Inputs:} Initial guess $(\bm{x}, \bm{\lambda}, \bm{\mu})$ and scalar parameters $\phi > 1$, $\bar{\lambda} > 0$, and $q$\\
        % Phase one
%		\STATE{}
		$\displaystyle \bm{\bar{\lambda}} = Proj_{[-\bar{\lambda}, \bar{\lambda}]} \left( \bm{\lambda} \right)$\\
		$\displaystyle \bm{x}' = \arg\min \left\{ \mathcal{L}_{q} \left( \bm{x}, \bm{\bar{\lambda}} \right) : \bm{x} \in \Omega \right\}$\\
		Set $\bm{\lambda}' \gets \bm{\bar{\lambda}} + 2 q \bm{h}(\bm{x}')$\\
		Construct $\bm{\mu}(\bm{x}', 1)$ from PPROJ output\footnote{Using formula (B.20) in companion paper \cite{diffenderfer2020local}}  and set $\bm{\mu}' \gets \bm{\mu}(\bm{x}', 1)$\\
%    	\STATE{}
		{\bfseries Return} $(\bm{x}', \bm{\lambda}', \bm{\mu}')$
%	\end{algorithmic}
\end{algorithm}
\end{center}

\textbf{Global Step (GS).} %Phase one serves to push the iterates towards a stationary point in the case that phase two is not creating a sufficient decrease in the global error estimator, $E_1$. 
The goal here is to compute a new iterate $(\bm{x}_{k + 1}, \bm{\lambda}_{k + 1}, \bm{\mu}_{k + 1})$ such that $E_{m,0} (\bm{x}_{k + 1}, \bm{\lambda}_{k + 1}, \bm{\mu}_{k + 1}) \leq \theta E_c (\bm{x}_{k + 1})$. This is done by using a method of multipliers approach together with PASA \cite{HagerActive}. Specifically, we use $\bm{x}_k$ as a starting guess in an iterative method applied to the polyhedral constrained augmented Lagrangian problem 
\begin{align}
\begin{array}{cc}
\min & \mathcal{L}_q (\bm{x}, \bm{\bar{\lambda}}_k) \\
\text{s.t.} & \bm{x} \in \Omega 
\end{array} \label{prob:NPASA-phase-1}
\end{align}
where the objective function in problem (\ref{prob:NPASA-phase-1}) is the augmented Lagrangian function given by 
\begin{align}
\mathcal{L}_q (\bm{x}, \bm{\lambda}) = f (\bm{x}) + \bm{\lambda}^{\intercal} \bm{h}(\bm{x}) + q \| \bm{h}(\bm{x}) \|^2, \label{def:augmented-Lagrangian}
\end{align}
$\bm{\bar{\lambda}}_k = Proj_{[-\bar{\lambda}, \bar{\lambda}]} \left( \bm{\lambda}_k \right)$, $\bar{\lambda} > 0$ is a safeguarding parameter, and $q \in \mathbb{R}$ is a penalty parameter. At a local minimizer of problem (\ref{prob:NPASA-phase-1}), there are easily computable multipliers for which $E_{m,0}$ vanishes which we use to update $\bm{\lambda}_{k+1}$ and $\bm{\mu}_{k+1}$. To encourage desired convergence properties, we use a technique often referred to as modified augmented Lagarangian with safeguarding. This approach has been useful in establishing convergence in several implementations of augmented Lagrangian solvers \cite{Birgin2005, Luo2008, Wu2009}. 
%Note that in the context of \cite{Luo2008}, the augmented Lagrangian function we are using corresponds to $L_2$ with $\theta(s) = \frac{1}{2} s^2$, which is the Rockafellar Lagrangian function. Additionally, as we are solving the problem with only the equality constraints in the augmented Lagrangian we remove the function $(\cdot)_+$ from the definition of $L_2$ in \cite{Luo2008}. Exact details of the approach used for the global step are provided in Section~\ref{section:npasa-phase-one}.  %and pseudocode for the GS algorithm in NPASA can be found in Algorithm~\ref{alg:gs}.

\begin{center}
\begin{algorithm}[!t]
%	\algsetup{linenosize=\notsotiny}
%	\notsotiny
	\caption{LS - Local Step Algorithm} \label{alg:cms} 
        % Initialize parameters
		{\bfseries Inputs:} Initial guess $(\bm{x}, \bm{\lambda}, \bm{\mu})$ and scalar parameters $\theta \in (0,1)$, $\alpha \in (0, 1]$, $\beta \geq 1$, $\sigma \in (0,1)$, $\tau \in (0,1)$, $p >> 1$, $\delta \in (0,1)$, $\gamma \in (0,1)$ \\ %, $i_C \in \mathbb{Z}_+$, $i_M \in \mathbb{Z}_+$\\
        % Phase two 
        \emph{Constraint Step}: Set $\bm{w}_0 \gets \bm{x}$ and set $i \gets 0$.\\
%        \For{$i = 0, 1, \ldots, i_C - 1$}{
        \While{$E_c (\bm{w}_i) > \theta E_{m,1} (\bm{x}, \bm{\lambda}, \bm{\mu})$}{
            Choose $p_i$ such that $p_i \geq \max \left\{ \beta^2, \| \bm{h} (\bm{w}_i) \|^{-2} \right\}$ and set $s_i = 1$\\
  			$\left[ \bm{\overline{w}}_{i+1}, \bm{y}_{i+1} \right] = \arg\min \left\{\| \bm{w} - \bm{w}_i \|^2 + p_i \| \bm{y} \|^2 : \nabla \bm{h} (\bm{w}_i) (\bm{w} - \bm{w}_i) + \bm{y} = - \bm{h} (\bm{w}_i), \bm{w} \in \Omega \right\}$\\
            Set $\alpha_{i+1} \gets 1 - \| \bm{y}_{i+1} \|$\\
            \eIf{$\alpha_{i+1} < \alpha$}{
                \textbf{Return} $(\bm{x}, \bm{\lambda}, \bm{\mu})$
            }{
            \While{$\| \bm{h} (\bm{w}_i + s_i (\bm{\overline{w}}_{i+1} - \bm{w}_{i})) \| > (1 - \tau \alpha_{i+1} s_i) \| \bm{h} (\bm{w}_i) \|$}{
                Set $s_i \gets \sigma s_i$
            }
    	    Set $\bm{w}_{i+1} \gets \bm{w}_i + s_i ( \bm{\overline{w}}_{i+1} - \bm{w}_i)$ and set $i \gets i+1$
            }
        }
   	    Set $\bm{w} \gets \bm{w}_i$\\
        \emph{Multiplier Step}: Set $\bm{z}_0 \gets \bm{w}$, set $p_0 \gets p$, and set $i \gets 0$\\
    	$(\bm{\nu}_0, \bm{\eta}_0) \in \arg \min \left\{ E_{m,0} ( \bm{z}_0, \bm{\nu}, \bm{\eta}) + \gamma \| [ \bm{\nu}, \bm{\eta} ] \|^2 : \bm{\eta} \geq \bm{0} \right\}$\\
    	$\bm{\eta}_0' \in \arg\min \left\{ E_{m,1} ( \bm{z}_0, \bm{\nu}_0, \bm{\eta}) : \bm{\eta} \geq \bm{0} \right\}$\\
        %Set $K_0 \gets E_{m,0} ( \bm{z}_0, \bm{\nu}_0, \bm{\eta}_0)$\\
%        \For{$i = 0, 1, \ldots, i_M - 1$}{
        \While{$E_{m,1} (\bm{z}_i, \bm{\nu}_i, \bm{\eta}_i') > \theta E_c (\bm{w})$}{
            Increase $p_i$ if necessary\\
            $\bm{z}_{i+1} = \arg \min \left\{ \mathcal{L}_{p_i}^i (\bm{z}, \bm{\nu}_i) : \nabla \bm{h} (\bm{z}_i) (\bm{z} - \bm{z}_i) = 0, \bm{z} \in \Omega \right\}$\\
    	    $(\bm{\nu}_{i+1}, \bm{\eta}_{i+1}) \in \arg\min \left\{ E_{m,0} ( \bm{z}_{i+1}, \bm{\nu}, \bm{\eta}) + \gamma \| [ \bm{\nu}, \bm{\eta} ] \|^2 : \bm{\eta} \geq \bm{0} \right\}$\\
            $\bm{\eta}_{i+1}' \in \arg\min \left\{ E_{m,1} ( \bm{z}_{i+1}, \bm{\nu}_{i+1}, \bm{\eta}) : \bm{\eta} \geq \bm{0} \right\}$\\
            %\eIf{$E_{m,1} (\bm{z}_{i+1}, \bm{\nu}_{i+1}, \bm{\eta}_{i+1}') \leq \delta K_i$}{
            %    Set $K_{i+1} \gets E_{m,1} (\bm{z}_{i+1}, \bm{\nu}_{i+1}, \bm{\eta}_{i+1}')$
            %}{
            %    Set $(\bm{\nu}_{i+1}, \bm{\eta}_{i+1}') \gets (\bm{\nu}_i, \bm{\eta}_i')$ and $K_{i+1} \gets K_i$
            %}
            \eIf{$E_{m,1} (\bm{z}_{i+1}, \bm{\nu}_{i+1}, \bm{\eta}_{i+1}') > \delta E_{m,1} (\bm{z}_i, \bm{\nu}_i, \bm{\eta}_i')$}{
                \textbf{Return} $(\bm{x}, \bm{\lambda}, \bm{\mu})$
            }{
                Set $i \gets i+1$
            }
            
        }
		{\bfseries Return} $(\bm{z}_i, \bm{\nu}_i, \bm{\eta}_i')$
\end{algorithm}
\end{center}

%Next, we consider the subproblems in phase two of NPASA which are chosen to achieve a balanced reduction of the constraint and multiplier error for problem (\ref{prob:main-nlp}). 
\textbf{Local Step (LS).} LS, Algorithm~\ref{alg:cms}, consists of two subproblems formulated to achieve a balanced reduction of the constraint and multiplier error for problem (\ref{prob:main-nlp}):
\begin{enumerate}
\item[1.] {\bfseries Constraint Step}. Use iterative method to generate $\bm{w}$ satisfying $E_c (\bm{w}) \leq \theta E_{m,1} (\bm{x}_k, \bm{\lambda}_k, \bm{\mu}_k)$.
\item[2.] {\bfseries Multiplier Step}. Starting from $\bm{w}$, apply an iterative method to generate $(\bm{z}, \bm{\nu}, \bm{\eta})$ which satisfies $E_{m,1} (\bm{z}, \bm{\nu}, \bm{\eta}) \leq \theta E_c (\bm{w})$.
\end{enumerate}
%Steps 1 and 2 are combined to form the LS algorithm contained in phase two of NAPSA. 
%Pseudocode for LS can be found in Algorithm~\ref{alg:cms}. 
To solve the constraint step we use a perturbed Newton's method scheme. With a unit step size, Newton's method is given by the iterative scheme 
\begin{align}
\bm{w}_{i + 1} = \arg \min \left\{ \| \bm{w} - \bm{w}_i \|^2 : \nabla \bm{h} (\bm{w}_i) (\bm{w} - \bm{w}_i ) = - \bm{h} (\bm{w}_i), \bm{w} \geq \bm{0} \right\}, \label{5.1}
\end{align}
where we start with $\bm{w}_0 = \bm{x}_k$. Problem (\ref{5.1}) can be solved using a polyhedral projection algorithm such as PPROJ \cite{Hager2016}. Based on the analysis in \cite{Hager1993}, the procedure in step 1 is locally quadratically convergent when the second-order sufficient optimality conditions hold and the active constraint gradients are linearly independent. However, as the constraint set in (\ref{5.1}) may be infeasible, we encourage robustness of this approach by introducing a perturbation to the linear constraint. The modified problem %by a vector $\bm{y}$ and add a quadratic penalty to the objective function so that the problem 
becomes
\begin{align}
\left[ \bm{w}_{i + 1}, \bm{y}_{i + 1} \right] = \arg \min \left\{ \| \bm{w} - \bm{w}_i \|^2 + p \| \bm{y} \|^2 : \nabla \bm{h} (\bm{w}_i) (\bm{w} - \bm{w}_i ) + \bm{y} = - \bm{h} (\bm{w}_i), \bm{w} \geq \bm{0} \right\}, \label{5.1p}
\end{align}
where we start with $\bm{w}_0 = \bm{x}_k$ and $p > 0$ is a penalty parameter. Under this formulation, this problem can still be solved using the algorithm PPROJ \cite{Hager2016} and we consider this method in detail in Subsection~\ref{subsec:infeas}. 
\iffalse %%%%%%%%%%%%%%%%%%%%%%%%%%%%%%%%%%%%%
We note that Gill and Wong discuss an alternative method for handling infeasibility \cite{Gill2010} that emulates adding a 1-norm penalty to the cost function instead of the square of the 2-norm which was used in (\ref{5.1p}). In their approach, elastic variables are introduced to create the problem
\begin{align}
\min \left\{ \| \bm{w} - \bm{w}_i \|^2 + \rho \ \bm{e}^{\intercal} (\bm{v} + \bm{y}) : \nabla \bm{h} (\bm{w}_i) (\bm{w} - \bm{w}_i ) - \bm{v} + \bm{y} = - \bm{h} (\bm{w}_i), \bm{w}, \bm{v}, \bm{y} \geq \bm{0} \right\}, \label{5.2p}
%\bm{y}_{i + 1} = \arg\min \left\{ \| \bm{z} - \bm{y}_i \|^2 + \rho \ \bm{e}^{\intercal} (\bm{v} + \bm{w}) : \nabla \bm{h} (\bm{y}_i) (\bm{z} - \bm{y}_i ) - \bm{v} + \bm{w} = - \bm{h} (\bm{y}_i), \bm{z}, \bm{v}, \bm{w} \in \mathbb{R}_{\geq 0} \right\}, \label{5.2p}
\end{align}
where $\bm{e}$ is a vector of ones and $\rho$ is a penalty parameter. It follows that if there exists a feasible solution to (\ref{5.1}) then for sufficiently large $\rho$ the solution to (\ref{5.2p}) is equal to the solution for (\ref{5.1}) (see Section 12.3 of \cite{Fletcher1987}). However, as the formulation in (\ref{5.2p}) is incompatible with PPROJ \cite{Hager2016}, NPASA makes use of the formulation in (\ref{5.1p}) for the constraint step. 
\fi %%%%%%%%%%%%%%%%%%%%%%%%%%%%%%%%%%%%%%%%%%

More details on the constraint step %in lines 2 -- 12 
can be found in Sections~\ref{subsec:infeas} and \ref{subsec:constraint-step-analysis}. The main problem can be %in line 5 and is 
solved using the PPROJ \cite{Hager2016} algorithm followed by a line search. %in lines 10 --12. 
The remaining lines of the constraint step are to ensure desired local convergence properties are satisfied. In particular, the requirements on $p_i$ and $\alpha_i$ are covered in detail during the convergence analysis in Section~\ref{section:npasa-phase-two}. As a note to the reader, %line 8 of 
when $\alpha_{i+1} < \alpha$ then Algorithm~\ref{alg:cms} results in phase two immediately branching to phase one in NPASA. This branching is triggered when the perturbation in the constraint step becomes too large thereby indicating that the constraint step may be experiencing trouble minimizing the constraint error, $E_c$. In such a case, NPASA works to find a better point by executing GS then branching back to LS.

Step 2, the multiplier step, can be decomposed into three parts:
\begin{itemize}
\item[2a.] {\bfseries Dual Step for equality multiplier}: $\displaystyle (\bm{\nu}_i, \bm{\eta}_i) \in \arg\min_{\bm{\lambda}, \bm{\eta}} \left\{ E_{m,0} (\bm{z}_i, \bm{\nu}, \bm{\eta}) + \gamma \| [ \bm{\nu}, \bm{\eta} ] \|^2: \bm{\eta} \geq \bm{0} \right\}$\\

\item[2b.] {\bfseries Dual Step for inequality multiplier}: $\displaystyle \bm{\eta}_i' \in \arg\min_{\bm{\eta}} \left\{ E_{m,1} (\bm{z}_i, \bm{\nu}_i, \bm{\eta}) : \bm{\eta} \geq \bm{0} \right\}$\\

\item[2c.] {\bfseries Primal Step}: $\displaystyle \bm{z}_{i + 1} = \arg\min_{\bm{z}} \left\{ \mathcal{L}_{p_i}^{i} (\bm{z}, \bm{\nu}_i) : \nabla \bm{h}(\bm{z}_i) (\bm{z} - \bm{z}_i) = \bm{0}, \bm{z} \in \Omega \right\}$,
\end{itemize}
where Step 2c makes use of the function defined by
\begin{align}
\mathcal{L}_{p}^{i} (\bm{z}, \bm{\nu}) := f(\bm{z}) + \bm{\nu}^{\intercal} \bm{h}(\bm{z}) + p \| \bm{h}(\bm{z}) - \bm{h}(\bm{z}_i) \|^2. \label{def:penalized-Lagrangian}
\end{align}
For initialization, $\bm{z}_0$ is set equal to the final $\bm{w}$ generated in the Constraint Step. 

Noting that $E_{m,0}$ is twice continuously differentiable with respect to $\bm{\lambda}$ and $\bm{\mu}$ and the hessian is positive semi-definite, there are many techniques available to solve Step 2a. The regularization term is added to ensure positive definiteness of the hessian which is used in establishing fast local convergence \cite{diffenderfer2020local}. We note that PASA \cite{HagerActive} can be used for Step 2a. The choice of $E_{m,0}$ in Step 2a over $E_{m,1}$ is required to establish desirable local convergence properties of NPASA for nondegenerate and degenerate minimizers of problem (\ref{prob:main-nlp}). More details can be found in the local convergence analysis of the multiplier step in the companion paper \cite{diffenderfer2020local}. 

The update to the inequality multiplier in Step 2b is required to achieve local quadratic convergence of the multiplier error $E_{m,1} (\bm{z}_i, \bm{\nu}_i, \bm{\eta}_i')$. While it seems that Step 2a should be able to minimize over $\bm{\nu}$ by fixing $\bm{\eta} = \bm{\eta}_{i-1}'$, there is difficulty in establishing a stability result for the solution in Step 2b that is required for local convergence analysis when assumption (\textbf{SCS}) fails to hold. Hence, the Step 2a and Step 2b structure is designed to ensure that NPASA still satisfies fast local convergence when approaching both degenerate and nondegenerate minimizers of (\ref{prob:main-nlp}). 

Step 2c is a polyhedral constrained optimization problem which can also be solved using PASA \cite{HagerActive}. It should be noted that in Step 2c the polyhedron will always contain a feasible point since the point $\bm{z} = \bm{z}_i$ satisfies all the constraints. Hence, no reformulation of the problem in Step 2c with some perturbation of the equality constraint is required to ensure robustness. A result on the global convergence properties of the multiplier step is found in Section~\ref{subsec:multiplier-step-analysis}. %Lastly, we note that the multiplier step in Algorithm~\ref{alg:cms} requires minimizing $E_{m,1}$ with respect to only the inequality multiplier variables. 
\textbf{NPASA.} As previously noted, NPASA, Algorithm~\ref{alg:npasa}, is split into two phases and contains criterion for branching between the phases that encourage desirable global and local convergence properties. Phase one contains the GS algorithm, updates for the penalty parameter used in GS, and criterion for branching to phase two. Phase two contains the LS algorithm and criterion for branching to phase one. The goal of the branching criterion is to remain in phase two whenever possible as fast local convergence is achieved there but to still ensure global convergence by branching to phase one when sufficient reduction of the error estimator is not occurring in phase two. More details on how these branching conditions achieve this goal can be found in Section~\ref{section:npasa-global-convergence}.

If the starting guess is good enough then NPASA can perform repeated iterations of the LS algorithm without branching to phase one. We continue in this manner until the error estimator $E_1$ is below the stopping tolerance, $\varepsilon$, at which point NPASA terminates and returns the final iterate. However, if the starting guess is not good or if the Newton iteration in the constraint step does not converge quickly then we branch to phase one of Algorithm~\ref{alg:npasa} and run the GS algorithm. Note that upon entering phase one we increase the penalty parameter $q_k$. As $q_k$ increases, $E_c$ typically decreases since the constraint is penalized in the objective function while $E_m$ decreases as iterates approach a stationary point of (\ref{prob:NPASA-phase-1}). 

\begin{center}
\begin{algorithm}[!t]
%	\algsetup{linenosize=\notsotiny}
%	\notsotiny
	\caption{NPASA - Nonlinear Polyhedral Active Set Algorithm} \label{alg:npasa} 
    % Initialize parameters
	{\bfseries Inputs:} Initial guess $(\bm{x}_{0}, \bm{\lambda}_{0}, \bm{\mu}_{0})$ and scalar parameters $\varepsilon > 0$, $\theta \in (0,1)$, $\phi > 1$, $\bar{\lambda} > 0$, $q_0 \geq 1$, $\alpha \in (0, 1]$, $\beta \geq 1$, $\sigma \in (0,1)$, $\tau \in (0,1)$, $p >> 1$, $\delta \in (0,1)$, $\gamma > 0$. \\%, $i_C \in \mathbb{Z}_+$, $i_M \in \mathbb{Z}_+$.\\
	Set $e_0 = E_1 (\bm{x}_{0}, \bm{\lambda}_{0}, \bm{\mu}_{0})$, $k = 0$, and goto phase one.\\
    % Phase one
	\textbf{------ Phase one ------}\\
	Set $q_k \gets \max \left\{ \phi, (e_{k-1})^{-1} \right\}  q_{k-1}$.\\
	\While{$E_{m,1} (\bm{x}_k, \bm{\lambda}_k, \bm{\mu}(\bm{x}_k, 1)) > \varepsilon$}{
	    \emph{Global Step}: $\displaystyle \left( \bm{x}_{k+1}, \bm{\lambda}_{k+1}, \bm{\mu}_{k+1} \right) \gets GS \left( \bm{x}_k, \bm{\lambda}_k, \bm{\mu}_k; \phi, \bar{\lambda}, q_k \right)$\\
		\emph{Update parameters}: $q_{k+1} \gets q_k$, $e_{k+1} \gets \min \left\{ E_1 (\bm{x}_{k+1}, \bm{\lambda}_{k+1}, \bm{\mu}_{k+1}), e_k \right\}$, $k \gets k+1$.\\
		\emph{Check branching criterion}:\\
        \If{$E_{m,1} (\bm{x}_k, \bm{\lambda}_k, \bm{\mu}_k) \leq \theta E_c (\bm{x}_{k-1})$}{
		    \textbf{goto} phase two
        }
    }
    % Phase two 
	\textbf{------ Phase two ------}\\
	\While{$E_{m,1} (\bm{x}_k, \bm{\lambda}_k, \bm{\mu}_k) > \varepsilon$}{
	    \emph{Local Step}: $\displaystyle \left( \bm{z}, \bm{\nu}, \bm{\eta} \right) \gets LS \left( \bm{x}_k, \bm{\lambda}_k, \bm{\mu}_k; \theta, \alpha, \beta, \sigma, \tau, p, \gamma, \delta \right)$\\
		\emph{Check branching criterion}:\\
        \eIf{$E_1 (\bm{z}, \bm{\nu}, \bm{\eta}) > \theta E_1 (\bm{x}_k, \bm{\lambda}_k, \bm{\mu}_k)$}{
           	\textbf{goto} phase one
        }{
            Set $(\bm{x}_{k+1}, \bm{\lambda}_{k+1}, \bm{\mu}_{k+1}) \gets (\bm{z}, \bm{\nu}, \bm{\eta})$ and set $e_{k+1} \gets \min \left\{ E_1 (\bm{x}_{k+1}, \bm{\lambda}_{k+1}, \bm{\mu}_{k+1}), e_k \right\}$\\
            Set $k \gets k+1$.
        }
    }
    {\bfseries Return} $(\bm{x}_k, \bm{\lambda}_k, \bm{\mu}_k)$
\end{algorithm}
\end{center}

Algorithm~\ref{alg:npasa} contains several parameters which we briefly outline here. First, we discuss $\varepsilon$ and $\theta$ as they are used in both phases. $\varepsilon$ is the stopping tolerance used in NPASA and $\theta$ is a scaling parameter used when checking if NPASA should branch to another phase. The choice of $\theta \in (0, 1)$ ensures that NPASA is decreasing the error estimator a sufficient amount at each iteration and is used to establish global convergence of NPASA. Phase one requires the parameters $\phi$, $\bar{\lambda}$, and $q_0$. The parameter $\phi$ is used to increase the penalty $q_k$ each time NPASA enters phase one and $q_0$ is the initial value for the penalty parameter in the augmented Lagrangian. A typical choice would be $\phi = 10$ so that the penalty increases by a factor of 10 each time NPASA branches to phase one. $\bar{\lambda}$ is used to safeguard the method of multipliers approach used in the GS algorithm in phase one. The LS algorithm in phase two requires the parameters $\beta$, $\alpha$, $\sigma$, $\tau$, $p$, $\gamma$, and $\delta$. $\beta$ is used to ensure that the penalty parameter in constraint step does not decrease at each iteration. $\alpha$ is used to determine how much perturbation is permitted at a solution of the constraint step. If $\alpha = 1$ then no perturbation is allowed. As $\alpha$ approaches zero, larger amounts of perturbation are permitted. %A safe and balanced choice is $\alpha = 1/2$. 
More details on the role of $\alpha$ can be found in Lemma~\ref{lem:Lem6.1}. The parameter $\tau$ is used during the line search in the constraint step. $\sigma$ is used to decrease the step length during the line search in the constraint step. A typical choice is $\sigma = 1/2$. $p$ is used to reset the penalty parameter at the start of the multiplier step in phase two. Updates to $p_i$ in the constraint and multiplier steps will be handled determinstically based on convergence analysis so that the user only needs to provide an initial value for $p$. $\gamma$ is used to ensure that the hessian of the problem in step 2a is positive definite.  Lastly, $\delta$ is used to determine if the multiplier error estimator in the multiplier step of LS has sufficiently decreased at each iteration. %Lastly, the positive integers $i_C$ and $i_M$ are the maximum number of iterations that NPASA will remain in the constraint and multiplier steps, respectively.

\begin{center}
\begin{algorithm}[!t]
%	\algsetup{linenosize=\notsotiny}
%	\notsotiny
	\caption{Modified PASA Algorithm for GS Augmented Lagrangian Problem} \label{alg:pasa-aug-lag} 
    % Initialize parameters
	{\bfseries Inputs:} Initial guess $\bm{x}_k$, scalars $\theta \in (0,1)$, $\gamma \in (0,1)$, $\varepsilon > 0$, $q_k > 0$ and vector parameter $\bm{\bar{\lambda}}_k$\\
	Set $\bm{u}_{0} = \bm{x}_{k}$, $\bm{u}_1 = \mathcal{P}_{\Omega} (\bm{u}_0)$, and $i = 1$\\
    % Phase one
	\textbf{------ Phase one ------}\\
    \While{$E_{m,0} (\bm{u}_i, \bm{\bar{\lambda}}_k + 2 q_k \bm{h} (\bm{u}_i), \bm{\mu}(\bm{u}_i,1)) > \varepsilon$}{
        Set $\bm{u}_{i+1} \gets GPA(\bm{u}_i)$ and $i \gets i+1$\\
        \If{$\mathcal{U} (\bm{u}_i) = \emptyset$ \textbf{and} $e_{\scaleto{PASA}{4pt}} (\bm{u}_i) < \theta E_{m,0} (\bm{u}_i, \bm{\bar{\lambda}}_k + 2 q_k \bm{h} (\bm{u}_i), \bm{\mu}(\bm{u}_i,1))$}{
            Set $\theta \gets \gamma \theta$
        }
        \If{$e_{\scaleto{PASA}{4pt}} (\bm{u}_i) \geq \theta E_{m,0} (\bm{u}_i, \bm{\bar{\lambda}}_k + 2 q_k \bm{h} (\bm{u}_i), \bm{\mu}(\bm{u}_i,1))$}{
    	    \textbf{goto} phase two
        }
    }
    % Phase two 
	\textbf{------ Phase two ------}\\
    \While{$E_{m,0} (\bm{u}_i, \bm{\bar{\lambda}}_k + 2 q_k \bm{h} (\bm{u}_i), \bm{\mu}(\bm{u}_i,1)) > \varepsilon$}{
        Set $\bm{u}_{i+1} \gets LCO(\bm{u}_i)$ and $i \gets i+1$\\
        \If{$e_{\scaleto{PASA}{4pt}} (\bm{u}_i) < \theta E_{m,0} (\bm{u}_i, \bm{\bar{\lambda}}_k + 2 q_k \bm{h} (\bm{u}_i), \bm{\mu}(\bm{u}_i,1))$}{
       	    \textbf{goto} phase one
        }
    }
    \textbf{Return} $\bm{u}_i$
\end{algorithm}
\end{center}

\subsection{Modified PASA Algorithm for Augmented Lagrangian Problem in Global Step} \label{sec:mod-pasa-alg}
Here, we provide Algorithm~\ref{alg:pasa-aug-lag} which is PASA with modified stopping and branching criterion designed specifically for the augmented Lagrangian problem found in line 3 of Algorithm~\ref{alg:gs}. %The only change here is that we have modified the stopping criterion. 
Note that the algorithms GPA and LCO used in Algorithm~\ref{alg:pasa-aug-lag} refer to a \emph{Gradient Projection Algorithm} and a \emph{Linearly Constrained Optimizer}, respectively. The original paper on PASA \cite{HagerActive} provides pseudocode for GPA and conditions that the LCO must satisfy in order for PASA to satisfy their convergence results. The details of the GPA and LCO used in Algorithm~\ref{alg:pasa-aug-lag} are the same as specified in the original PASA paper. The value $e_{\scaleto{PASA}{4pt}} (\bm{x})$ corresponds to the \emph{local error estimator for PASA} as defined in \cite{HagerActive}. 
%but is also provided in (\ref{def:e-PASA}) during the analysis of the GS Algorithm using PASA in Section~\ref{subsec:gs-pasa}.

Traditionally, the stopping criterion in PASA at the iterate $\bm{x}$ is given by $E_{\scaleto{PASA}{4pt}} (\bm{x}) \leq \varepsilon$, where $\varepsilon$ is a user provided stopping tolerance value and $E_{\scaleto{PASA}{4pt}} (\bm{x})$ denotes the \emph{global error estimator for PASA} at the primal variable $\bm{x}$ originally defined in \cite{HagerActive}. 
%but also provided in equation (\ref{def:E-PASA}) in Section~\ref{subsec:gs-pasa} during our convergence analysis. 
We have modified this criterion by replacing $E_{\scaleto{PASA}{4pt}}$ with the multiplier error estimator for problem (\ref{prob:main-nlp}), $E_{m,0}$, and intend to use the formula derived during the following analysis, in (\ref{eq:aug-lag-mod.4}), to compute $E_{m,0}$. The motivation for this modified stopping and branching criterion is to reduce the number of assumptions required to ensure convergence of Algorithm~\ref{alg:gs}. While global convergence for Algorithm~\ref{alg:gs} can be established when using PASA with the original stopping criterion, 
%in Section~\ref{subsec:gs-pasa}, 
the number of assumptions required for convergence are reduced when using Algorithm~\ref{alg:pasa-aug-lag}. A more detailed analysis yielding the modifications can be found in Section~\ref{subsec:gs-mod-pasa}.

\section{Convergence Analysis for Phase One of NPASA} \label{section:npasa-phase-one}
We begin our convergence analysis by performing convergence analysis for phase one of NPASA. Motivation for the global step in phase one is to ensure the robustness of the NPASA algorithm. In particular, phase two of NPASA may encounter problems when the linearization of $\bm{h} (\bm{x})$ at the current iterate requires a large perturbation to be feasible resulting in an unsuccessful minimization of $\| \bm{h} (\bm{x}) \|$ in the Constraint Step. In such an instance, phase one serves as a safeguard to move the primal iterate to a new point before attempting phase two again. For phase one we have decided to use augmented Lagrangian techniques. First introduced by Hestenes \cite{Hestenes1969} and Powell \cite{Powell1969} and sometimes referred to as the method of multipliers, augmented Lagrangian methods and their convergence properties have since been well studied \cite{bertsekas1995nonlinear, Luo2008, NandW} and have been successful at solving nonlinear programs in practice using software packages such as ALGENCAN \cite{Andreani2008}, ALGLIB \cite{Bochkanov1999}, LANCELOT \cite{Conn1991, Conn2010}, and MINOS \cite{Murtagh1978}. Hence, the choice of an augmented Lagrangian technique is a suitable safeguard to ensure NPASA will converge to a stationary point in the case that problems arise during the constraint step in phase two.

This section is organized as follows. First, we analyze the method of multipliers scheme in %line 3 of 
Algorithm~\ref{alg:gs}, GS algorithm. In particular, we establish necessary conditions under which a subsequence of iterates of the GS algorithm converge to a stationary point for problem (\ref{prob:main-nlp}). During this discussion, we also develop a closed form expression for computing the multiplier error of the original nonlinear program, problem (\ref{prob:main-nlp}), when solving the augmented Lagrangian problem without requiring the explicit value of the inequality multiplier. 
%In Section~\ref{subsec:gs-pasa}, we provide necessary conditions under which the GS algorithm converges to a stationary point when using PASA \cite{HagerActive} to solve the augmented Lagrangian problem in Algorithm~\ref{alg:gs}. 
In Section~\ref{subsec:gs-mod-pasa}, we provide theoretical analysis for Algorithm~\ref{alg:pasa-aug-lag}, the modified PASA algorithm, which reduces the number of necessary conditions required to establish convergence when solving the minimization problem in the GS algorithm using Algorithm~\ref{alg:pasa-aug-lag}.
%Finally, in Section~\ref{subsec:gs-inequality-multiplier} we discuss how to construct the inequality multiplier required at the end of the GS algorithm.

\subsection{Preliminary Convergence Analysis for GS Algorithm} \label{subsec:gs-conv-analysis}

Consider the minimization problem in %line 3 of 
Algorithm~\ref{alg:gs} at the $k$th iteration of NPASA given by
\begin{align}
\begin{array}{cc}
\displaystyle \min_{\bm{x} \in \mathbb{R}^{n}} & \mathcal{L}_{q_k} (\bm{x}, \bm{\bar{\lambda}}_k) = f (\bm{x}) + \bm{\bar{\lambda}}_k^{\intercal} \bm{h} (\bm{x}) + q_k \| \bm{h} (\bm{x}) \|^2 \\
\text{s.t.} & \bm{r}(\bm{x}) \leq \bm{0}
\end{array} \label{prob:aug-lag-mod}
\end{align}
where $\bm{r} (\bm{x}) = \bm{A} \bm{x} - \bm{b}$ is the linear function from problem (\ref{prob:main-nlp}) and $\bm{\bar{\lambda}}_k = Proj_{[-\bar{\lambda}, \bar{\lambda}]} (\bm{\lambda}_k)$ is known. Now suppose that $\bm{x}_{k+1}$ is a local minimizer for problem (\ref{prob:aug-lag-mod}). Then by the KKT conditions there exists a $\bm{\mu}_{k+1} \geq \bm{0}$ such that
\begin{align}
\bm{0} 
&= \nabla_x \mathcal{L}_{q_k} (\bm{x}, \bm{\bar{\lambda}}_k) + \bm{\mu}_{k+1}^{\intercal} \bm{A}
%&= \nabla f (\bm{x}_{k+1}) + (\bm{\bar{\lambda}}_k + 2 q_k \bm{h} (\bm{x}_{k+1}))^{\intercal} \nabla \bm{h} (\bm{x}_{k+1}) + \bm{\mu}_{k+1}^{\intercal} \bm{A} \\
= \nabla \mathcal{L} (\bm{x}_{k+1}, \bm{\bar{\lambda}}_k + 2 q_k \bm{h} (\bm{x}_{k+1}), \bm{\mu}_{k+1})
\end{align}
and $\bm{\mu}_{k+1}^{\intercal} \bm{r} (\bm{x}_{k+1}) = 0$. Using the formula 
%\begin{align}
$\bm{\lambda}_{k+1} = \bm{\bar{\lambda}}_k + 2 q_k \bm{h} (\bm{x}_{k+1})$
%\end{align}
from line 4 of Algorithm~\ref{alg:gs} it now follows from the definition of $E_{m,0}$ that
%\begin{align}
$E_{m,0} (\bm{x}_{k+1}, \bm{\lambda}_{k+1}, \bm{\mu}_{k+1}) = 0$.
%\end{align}
Hence, we have that the multiplier error is equal to zero at $(\bm{x}_{k+1}, \bm{\lambda}_{k+1}, \bm{\mu}_{k+1})$, where $\bm{\mu}_{k+1}$ is some vector that we do not explicitly know. 
%The global error estimator used for stopping criterion in PASA at a point $\bm{x} \in \Omega$ is given by $E_{\scaleto{PASA}{4pt}} (\bm{x}) = \| \bm{y} (\bm{x},1) - \bm{x} \|$ where 

Now define $\bm{y} (\bm{x},\alpha)$ to be the minimizer for the problem
\begin{align}
\begin{array}{cc}
\displaystyle \min_{\bm{y}} & \displaystyle \frac{1}{2} \| \bm{x} - \alpha \bm{g} (\bm{x}) - \bm{y} \|^2 \\
\text{s.t.} & \bm{r}(\bm{y}) \leq \bm{0}
\end{array} \label{prob:pasa-error}
\end{align} 
for $\alpha > 0$ and where $\bm{g} (\bm{x})$ is the transpose of the gradient of the cost function in 
%the polyhedral constrained problem being solved by PASA. 
problem (\ref{prob:aug-lag-mod}). 
%So if we are solving problem (\ref{prob:aug-lag-mod}) using PASA, note 
Note that in our case we have 
\begin{align}
\bm{g} (\bm{x})^{\intercal} 
&= \nabla_x \mathcal{L}_{q_k} (\bm{x}, \bm{\bar{\lambda}}_k) 
= \nabla f (\bm{x}) + \bm{\bar{\lambda}}_k^{\intercal} \nabla \bm{h} (\bm{x}) + 2 q_k \bm{h} (\bm{x})^{\intercal} \nabla \bm{h} (\bm{x}), \label{eq:aug-lag-mod.0.0}
\end{align} 
however we will interchangeably write $\bm{g} (\bm{x})^{\intercal}$ to simplify notation. 
%When an iterate $\bm{x} \in \Omega$ of PASA results in $E_{\scaleto{PASA}{4pt}} (\bm{x})$ being sufficiently small, where the tolerance is set by the user, then PASA terminates. 
Now fix $\bm{x} \in \Omega$ and consider problem (\ref{prob:pasa-error}). By the KKT conditions, there exists a vector $\bm{\mu}(\bm{x},\alpha)$ such that
\iffalse
\begin{enumerate}[leftmargin=2\parindent,align=left,labelwidth=\parindent,labelsep=7pt]
\item[(\ref{prob:pasa-error}.1)] Gradient of Lagrangian equals zero: $(\bm{y}(\bm{x},\alpha) - \bm{x})^{\intercal} + \alpha \bm{g} (\bm{x})^{\intercal} + \bm{\mu}(\bm{x}, \alpha)^{\intercal} \bm{A} = \bm{0}$
\item[(\ref{prob:pasa-error}.2)] Satisfies inequality constraints: $\bm{r}(\bm{y}(\bm{x},\alpha)) \leq \bm{0}$
\item[(\ref{prob:pasa-error}.3)] Nonnegativity of inequality multipliers: $\bm{\mu} (\bm{x},\alpha) \geq \bm{0}$
\item[(\ref{prob:pasa-error}.4)] Complementary slackness: $\mu_i (\bm{x},\alpha) r_i(\bm{y}(\bm{x},\alpha)) = 0$, for $1 \leq i \leq m$.
\end{enumerate}
\fi
\begin{align}
&\text{Gradient of Lagrangian equals zero:} \ (\bm{y}(\bm{x},\alpha) - \bm{x})^{\intercal} + \alpha \bm{g} (\bm{x})^{\intercal} + \bm{\mu}(\bm{x}, \alpha)^{\intercal} \bm{A} = \bm{0},
\label{prob:pasa-error.1} \\
&\text{Satisfies inequality constraints:} \ \bm{r}(\bm{y}(\bm{x},\alpha)) \leq \bm{0}, \label{prob:pasa-error.2} \\
&\text{Nonnegativity of inequality multipliers:} \ \bm{\mu} (\bm{x},\alpha) \geq \bm{0}, \label{prob:pasa-error.3} \\
&\text{Complementary slackness:} \  \mu_i (\bm{x},\alpha) r_i(\bm{y}(\bm{x},\alpha)) = 0, \ \text{for} \ 1 \leq i \leq m. \label{prob:pasa-error.4}
\end{align}
Our goal is to use the KKT conditions (\ref{prob:pasa-error.1}) -- (\ref{prob:pasa-error.4}) for problem (\ref{prob:pasa-error}) to derive a formula for the multiplier error of our original problem, $E_{m,0}$, in terms of the current 
%iterate of PASA. 
primal iterate, $\bm{x}$. 
%Then we will show that we can bound this expression for $E_{m,0}$ in terms of the global error estimator of PASA. 
This idea will be useful in establishing a theoretical convergence result for phase one of NPASA and for later establishing practical convergence results using PASA and modified PASA. This is summarized in the following result.
\begin{lemma} \label{lem:Em0-Epasa-comparison}
Suppose $\bm{x} \in \Omega$. Then for each $\alpha > 0$ there exists an inequality multiplier $\bm{\mu} (\bm{x},\alpha)$ such that 
\begin{align}
E_{m,0} (\bm{x}, \bm{\bar{\lambda}}_k + 2 q_k \bm{h} (\bm{x}), \bm{\mu} (\bm{x},\alpha)/\alpha) 
&= - \bm{g}(\bm{x})^{\intercal} (\bm{y}(\bm{x},\alpha) - \bm{x}) + \left(\frac{1}{\alpha^2} - \frac{1}{\alpha} \right) \| \bm{y}(\bm{x},\alpha) - \bm{x} \|^2. \label{eq:aug-lag-mod.4}
\end{align}
In particular, for $\alpha = 1$ we have
\begin{align}
\| \bm{y}(\bm{x},1) - \bm{x} \|^2 
&\leq E_{m,0} (\bm{x}, \bm{\bar{\lambda}}_k + 2 q_k \bm{h} (\bm{x}), \bm{\mu} (\bm{x},1)) 
\leq \| \nabla_x \mathcal{L}_{q_k} (\bm{x}, \bm{\bar{\lambda}}_k) \| \| \bm{y}(\bm{x}, 1) - \bm{x} \|. \label{eq:aug-lag-mod.5}
\end{align}
%and $E_{m,0} (\bm{x}, \bm{\bar{\lambda}}_k + 2 q_k \bm{h} (\bm{x}), \bm{\mu} (\bm{x},1)) = 0$ if and only if $\bm{x}$ is a stationary point for problem (\ref{prob:aug-lag-mod}).
\end{lemma}
\begin{proof}
Let $\bm{x} \in \Omega$. By recalling that the multiplier error estimator $E_{m,0}$ for problem (\ref{prob:main-nlp}) is given by $E_{m,0} (\bm{x}, \bm{\lambda}, \bm{\mu}) = \| \nabla f (\bm{x}) + \bm{\lambda}^{\intercal} \nabla \bm{h} (\bm{x}) + \bm{\mu}^{\intercal} \bm{A} \|^2 - \bm{\mu}^{\intercal} \bm{r} (\bm{x})$, it follows that 
\begin{align}
E_{m,0} (\bm{x}, \bm{\bar{\lambda}}_k + 2 q_k \bm{h} (\bm{x}), \bm{\mu}) 
&= \| \bm{g} (\bm{x})^{\intercal} + \bm{\mu}^{\intercal} \bm{A} \|^2 - \bm{\mu}^{\intercal} \bm{r} (\bm{x}). \label{eq:aug-lag-mod.0}
\end{align} 
Now note that from (\ref{prob:pasa-error.1}) we have
%\begin{align}
$\bm{g} (\bm{x})^{\intercal} + \frac{1}{\alpha} \bm{\mu} (\bm{x},\alpha)^{\intercal} \bm{A} = \frac{1}{\alpha}(\bm{x} - \bm{y}(\bm{x},\alpha))^{\intercal}$.
%\end{align}
Taking the norm of both sides and squaring yields
\begin{align}
\left\| \bm{g} (\bm{x})^{\intercal} + \frac{1}{\alpha} \bm{\mu} (\bm{x},\alpha)^{\intercal} \bm{A} \right\|^2 = \frac{1}{\alpha^2} \| \bm{y}(\bm{x},\alpha) - \bm{x} \|^2. \label{eq:aug-lag-mod.1}
\end{align}
Additionally, right multiplying both sides of (\ref{prob:pasa-error.1}) by $\frac{1}{\alpha} (\bm{y}(\bm{x},\alpha) - \bm{x})$ and using (\ref{prob:pasa-error.4}) yields
\begin{align}
- \frac{1}{\alpha} \bm{\mu}(\bm{x},\alpha)^{\intercal} \bm{r}(\bm{x}) = - \frac{1}{\alpha} \| \bm{y}(\bm{x},\alpha) - \bm{x} \|^2 - \bm{g}(\bm{x})^{\intercal} (\bm{y}(\bm{x},\alpha) - \bm{x}). \label{eq:aug-lag-mod.2}
\end{align}
%Note that as $\bm{x} \in \Omega$ we have that $-\bm{\mu}(\bm{x})^{\intercal} \bm{A} \bm{x} \geq 0$ which implies that $\nabla_x \mathcal{L}_{q} (\bm{x}, \bm{\bar{\lambda}}_k) (\bm{y}(\bm{x}) - \bm{x}) \leq \bm{0}$. 
Combining (\ref{eq:aug-lag-mod.0}), (\ref{eq:aug-lag-mod.1}), and (\ref{eq:aug-lag-mod.2}) yields (\ref{eq:aug-lag-mod.4}). 

Next, from P6 in Proposition 2.1 of \cite{Hager2006} we have that 
%\begin{align}
$- \bm{g} (\bm{x})^{\intercal} (\bm{y}(\bm{x},\alpha) - \bm{x}) \geq \frac{1}{\alpha} \| \bm{y}(\bm{x},\alpha) - \bm{x} \|^2$,
%\end{align}
for all $\bm{x} \in \Omega$. Combining this fact with (\ref{eq:aug-lag-mod.4}) yields a lower bound for the multiplier error
\begin{align}
E_{m,0} (\bm{x}, \bm{\bar{\lambda}}_k + 2 q_k \bm{h} (\bm{x}), \bm{\mu} (\bm{x},\alpha)/\alpha) 
\geq \frac{1}{\alpha^2} \| \bm{y}(\bm{x},\alpha) - \bm{x} \|^2. \label{eq:aug-lag-mod.4.1}
\end{align}
In particular, by taking $\alpha = 1$ in (\ref{eq:aug-lag-mod.4}) and recalling (\ref{eq:aug-lag-mod.0.0}) we have
\begin{align}
E_{m,0} (\bm{x}, \bm{\bar{\lambda}}_k + 2 q_k \bm{h} (\bm{x}), \bm{\mu} (\bm{x},1)) 
= - \nabla_x \mathcal{L}_{q_k} (\bm{x}, \bm{\bar{\lambda}}_k) (\bm{y}(\bm{x},1) - \bm{x}). \label{eq:aug-lag-mod.4.2}
\end{align}
Lastly, (\ref{eq:aug-lag-mod.4.1}) and (\ref{eq:aug-lag-mod.4.2}) 
%with the definition of $E_{\scaleto{PASA}{4pt}} (\bm{x})$ and 
together with
Cauchy-Schwarz yield (\ref{eq:aug-lag-mod.5}). 
%
%Lastly, as $E_{\scaleto{PASA}{4pt}} (\bm{x}) = 0$ if and only if $\bm{x}$ is a stationary point \cite{HagerActive}, we conclude from (\ref{eq:aug-lag-mod.5}) that $E_{m,0} (\bm{x}, \bm{\bar{\lambda}}_k + 2 q_k \bm{h} (\bm{x}), \bm{\mu} (\bm{x},1)) = 0$ if and only if $\bm{x}$ is a stationary point. 
\end{proof}

Hence, in (\ref{eq:aug-lag-mod.4}) we have derived a closed form expression for the multiplier error estimator for the general nonlinear 
%program in terms of the current iterate of PASA and a known KKT equality multiplier.
program, problem (\ref{prob:main-nlp}), in terms of the current primal iterate, $\bm{x}$, and a known KKT equality multiplier. 
%Additionally, in (\ref{eq:aug-lag-mod.5}) upper and lower bounds for the multiplier error estimator for the general nonlinear program were derived in terms of the global error estimator used by PASA. Thus, we can compute and bound the value of the multiplier error estimator $E_{m,0} (\bm{x}_k, \bm{\lambda}_k, \bm{\mu}(\bm{x}_k,\alpha))$ by using information already computed by PASA without requiring the inequality multiplier vector $\bm{\mu}(\bm{x}_k,\alpha)$. We now provide the following result which will be useful to note for practical implementations of this approach for augmented Lagrangian problems. The result follows immediately from (\ref{eq:aug-lag-mod.5}) and the definition of the dot product.
The inequality in (\ref{eq:aug-lag-mod.5}) will be useful in establishing the following convergence result. %for the GS algorithm.

\begin{theorem} \label{thm:pasa-aug-lag-global-conv}
Suppose $\Omega$ is compact and that $f, \bm{h} \in \mathcal{C}^1 (\Omega)$. Suppose the following are satisfied:
\begin{enumerate}
\item $\{ \bm{x}_k \}_{k=0}^{\infty}$ is a sequence with $\bm{x}_k \in \Omega$ and $E_{m,0} (\bm{x}_k, \bm{\lambda}_k, \bm{\mu}(\bm{x}_k,1)) \leq \theta E_c (\bm{x}_k)$, for all $k \geq 0$.
\item $\{ \bm{\lambda}_k \}_{k=0}^{\infty}$ is updated using the formula for $\bm{\lambda}'$ in line 4 of Algorithm~\ref{alg:gs} (GS algorithm).
\item $\{ q_k \}_{k=0}^{\infty}$ is a sequence of real numbers such that $q_k \to \infty$ as $k \to \infty$.
\end{enumerate}
Then any subsequence of $\{ \bm{x}_k \}_{k=0}^{\infty}$ with limit point $\bm{x}^*$ such that assumption (\textbf{LICQ}) holds at $\bm{x}^*$ satisfies $E_1 (\bm{x}^*, \bm{\lambda}^*, \bm{\mu}(\bm{x}^*,1)) = E_0 (\bm{x}^*, \bm{\lambda}^*, \bm{\mu}(\bm{x}^*,1)) = 0$, where $\bm{\lambda}^*$ is the limit of the corresponding subsequence for $\{ \bm{\lambda}_k \}$.
\end{theorem}
\begin{proof}
Let $\{ \bm{x}_k \}_{k=0}^{\infty}$ be a subsequence with limit point $\bm{x}^*$ and let $\{ \bm{\lambda}_k \}_{k=0}^{\infty}$ and $\{ \bm{\mu} (\bm{x}_k,1) \}_{k=0}^{\infty}$ be corresponding subsequences of multipliers. For simplicity, let $\mathcal{A} \:= \mathcal{A}(\bm{x}^*)$. First, we show that $\{ \bm{\lambda}_k \}_{k=0}^{\infty}$ and $\{ \bm{\mu} (\bm{x}_k,1) \}_{k=0}^{\infty}$ are uniformly bounded for $k$ sufficiently large. Rearranging the terms in equation (\ref{prob:pasa-error.1}) with $\alpha = 1$ combined with the fact that %As $\bm{x}_k \to \bm{x}^*$, for sufficiently large $k$ it follows that 
$\mathcal{A}(\bm{x}_k) \subseteq \mathcal{A}(\bm{x}^*)$ for sufficiently large $k$, we have that 
\begin{align}
\begin{bmatrix}
\nabla \bm{h} (\bm{x}_k)^{\intercal} \ \Big\vert \ \bm{A}_{\mathcal{A}}^{\intercal}
\end{bmatrix} \begin{bmatrix}
\bm{\lambda}_k \\
\bm{\mu}_{\mathcal{A}} (\bm{x}_k,1)
\end{bmatrix}
= \bm{x}_k - \bm{y}(\bm{x}_k,1) - \nabla f (\bm{x}_k)^{\intercal}
\end{align}
By our hypothesis that (\textbf{LICQ}) holds at $\bm{x}^*$, there exists a $N$ such that the matrix $\begin{bmatrix} \nabla \bm{h} (\bm{x}_k)^{\intercal} \ \vert \ \bm{A}_{\mathcal{A}}^{\intercal} \end{bmatrix}$ is of full column rank, for all $k \geq N$. Hence, there exists a constant $M$ such that
\begin{align}
\left\| \begin{bmatrix}
\bm{\lambda}_k \\
\bm{\mu}_{\mathcal{A}} (\bm{x}_k,1)
\end{bmatrix} \right\|
\leq M \| \bm{x}_k - \bm{y}(\bm{x}_k,1) - \nabla f (\bm{x}_k)^{\intercal} \|, \label{eq:phase-one-convergence.1}
\end{align}
for $k$ sufficiently large. From property P6 in Proposition 2.1 of \cite{Hager2006} we have
%\begin{align}
$\| \bm{x}_k - \bm{y}(\bm{x}_k,1) \|^2 
\leq \nabla_x \mathcal{L}_{q} (\bm{x}_k, \bm{\bar{\lambda}}_{k-1})^{\intercal} (\bm{x}_k - \bm{y}(\bm{x}_k,1))$.
%\end{align}
Hence, using the hypothesis that $E_{m,0} (\bm{x}_k, \bm{\lambda}_k, \bm{\mu}(\bm{x}_k,1)) \leq \theta E_c (\bm{x}_k)$ together with (\ref{eq:aug-lag-mod.5}) it follows that
%\begin{align}
$\| \bm{x}_k - \bm{y}(\bm{x}_k,1) \|^2 
\leq E_{m,0} (\bm{x}_k, \bm{\lambda}_k, \bm{\mu} (\bm{x}_k,1)) 
\leq E_c (\bm{x}_k)
= \| \bm{h} (\bm{x}_k) \|^2$.
%\end{align}
%As $E_c (\bm{x}_k) = \| \bm{h} (\bm{x}_k) \|^2$, 
Combining this with the inequality in (\ref{eq:phase-one-convergence.1}) yields
\begin{align}
\left\| \begin{bmatrix}
\bm{\lambda}_k \\
\bm{\mu}_{\mathcal{A}} (\bm{x}_k,1)
\end{bmatrix} \right\|
\leq M \left( \| \bm{h} (\bm{x}_k) \| + \| \nabla f (\bm{x}_k) \| \right), \label{eq:phase-one-convergence.2}
\end{align}
for $k$ sufficiently large. As $f, \bm{h} \in \mathcal{C}^1$ and $\mu_j(\bm{x}_k, 1) = 0$ for $j \not\in \mathcal{A}$, it follows from (\ref{eq:phase-one-convergence.2}) that the sequences $\{ \bm{\lambda}_k \}_{k=0}^{\infty}$ and $\{ \bm{\mu}(\bm{x}_k,1) \}_{k=0}^{\infty}$ are uniformly bounded for $k$ sufficiently large. As the sequence $\{ \bm{\bar{\lambda}}_k \}_{k=0}^{\infty}$ is uniformly bounded by construction, there exists a constant $\Lambda$ such that $\| \bm{\lambda}_{k} - \bm{\bar{\lambda}}_{k-1} \| \leq \Lambda$ for all $k$ sufficiently large. Now using the update formula $\bm{\lambda}_{k} = \bm{\bar{\lambda}}_{k-1} + 2 q_k \bm{h} (\bm{x}_{k})$, we have that
\begin{align}
\| \bm{h} (\bm{x}_{k}) \| 
= \frac{\| \bm{\lambda}_{k} - \bm{\bar{\lambda}}_{k-1} \|}{2 q_k} 
\leq \frac{\Lambda}{2 q_k}, \label{eq:phase-one-convergence.3}
\end{align}
for all $k$ sufficiently large. So (\ref{eq:phase-one-convergence.3}) together with the hypothesis that $q_k \to \infty$ as $k \to \infty$, yields that
\begin{align}
\lim_{k \to \infty} E_c (\bm{x}_{k}) = \lim_{k \to \infty} \| \bm{h} (\bm{x}_{k}) \|^2 \leq \lim_{k \to \infty} \frac{\Lambda^2}{4 q_k^2} = 0.
\end{align}
As $E_0 (\bm{x}_k, \bm{\lambda}_k, \bm{\mu} (\bm{x}_k,1))^2 = E_{m,0} (\bm{x}_k, \bm{\lambda}_k, \bm{\mu} (\bm{x}_k,1)) + E_c (\bm{x}_k)$ and $E_{m,0} (\bm{x}_k, \bm{\lambda}_k, \bm{\mu} (\bm{x}_k,1)) \leq \theta E_c (\bm{x}_k)$ for all $k$ by hypothesis, it now follows that $E_0 (\bm{x}^*, \bm{\lambda}^*, \bm{\mu} (\bm{x}^*,1)) = 0$. Lastly, from Lemma~\ref{lem:E0-E1-relationship} it follows that $E_1 (\bm{x}^*, \bm{\lambda}^*, \bm{\mu} (\bm{x}^*,1)) = 0$ which concludes the proof.
\end{proof}

\subsection{Global Convergence of GS Algorithm using Modified PASA} \label{subsec:gs-mod-pasa}
We now focus on establishing a global convergence result for the augmented Lagrangian problem when using Algorithm~\ref{alg:pasa-aug-lag}, PASA \cite{Hager2016} with modified stopping criterion. 
%To this end, we determine slight modifications to the PASA algorithm, Algorithm 2 in \cite{Hager2016}. 
Algorithm~\ref{alg:pasa-aug-lag} is obtained by replacing all instances of $E_{\scaleto{PASA}{4pt}} (\bm{x})$ in the original PASA with $E_{m,0} (\bm{x}, \bm{\bar{\lambda}}_k + 2 q_k \bm{h} (\bm{x}), \bm{\mu}(\bm{x},1)) = -\nabla \mathcal{L}_q (\bm{x}, \bm{\bar{\lambda}}_k) (\bm{y} (\bm{x},1) - \bm{x})$. First, we provide a necessary lemma that considers the relationship of the multiplier error estimator $E_{m,0} (\bm{x}, \bm{\bar{\lambda}}_k + 2 q_k \bm{h} (\bm{x}), \bm{\mu}(\bm{x},\alpha)/\alpha)$ to $E_{m,0} (\bm{x}, \bm{\bar{\lambda}}_k + 2 q_k \bm{h} (\bm{x}), \bm{\mu}(\bm{x},\beta)/\beta)$ for various values of $0 < \alpha \leq \beta$. 
\begin{lemma} \label{thm:aug-lag-mult-error}
For a given $\bm{x}$, let $\bm{\mu} (\bm{x}, \alpha)$ satisfy the KKT conditions in (\ref{prob:pasa-error.1}) -- (\ref{prob:pasa-error.4}). Then
\iffalse
\begin{enumerate}[leftmargin=2\parindent,align=left,labelwidth=\parindent,labelsep=7pt]
\item[(\ref{thm:aug-lag-mult-error}.1)] For $0 < \alpha \leq \beta$, $-\nabla \mathcal{L}_{q_k} (\bm{x}, \bm{\bar{\lambda}}_k) (\bm{y} (\bm{x},\alpha) - \bm{x}) 
\leq -\nabla_x \mathcal{L}_{q_k} (\bm{x}, \bm{\bar{\lambda}}_k) (\bm{y} (\bm{x},\beta) - \bm{x})$. \vspace{1mm}
%%%%%%%%%%%%%%%% NOT NEEDED FOR PAPER %%%%%%%%%%%%%%%%%%%%%%
\iffalse 
\item[(\ref{thm:aug-lag-mult-error}.2old)] For $0 < \alpha \leq \beta$, $- \alpha \nabla \mathcal{L}_q (\bm{x}, \bm{\bar{\lambda}}_k) (\bm{y} (\bm{x},\beta) - \bm{x}) 
\leq - \beta \nabla \mathcal{L}_q (\bm{x}, \bm{\bar{\lambda}}_k) (\bm{y} (\bm{x},\alpha) - \bm{x})$. \vspace{1mm}
\fi
%%%%%%%%%%%%%%%%%%%%%%%%%%%%%%%%%%%%%%%%%%%%%%%%%%%%%%%%%%%%
\item[(\ref{thm:aug-lag-mult-error}.2)] If $0 < \alpha \leq \beta \leq 1$ then 
\begin{align*}
E_{m,0} (\bm{x}, \bm{\bar{\lambda}}_k + 2 q_k \bm{h} (\bm{x}), \bm{\mu}(\bm{x},\beta)/\beta) 
\leq E_{m,0} (\bm{x}, \bm{\bar{\lambda}}_k + 2 q_k \bm{h} (\bm{x}), \bm{\mu}(\bm{x},\alpha)/\alpha).
\end{align*}
%%%%%%%%%%%%%%%% NOT NEEDED FOR PAPER %%%%%%%%%%%%%%%%%%%%%%
\iffalse 
\vspace{1mm}
\item[(\ref{thm:aug-lag-mult-error}.4)] If $1 \leq \alpha \leq \beta \leq 3$ then
\begin{align*}
\alpha E_{m,0} (\bm{x}, \bm{\bar{\lambda}}_k + 2 q \bm{h} (\bm{x}), \bm{\mu}(\bm{x},\beta)/\beta) 
\leq \beta E_{m,0} (\bm{x}, \bm{\bar{\lambda}}_k + 2 q \bm{h} (\bm{x}), \bm{\mu}(\bm{x},\alpha)/\alpha).
\end{align*}
\fi
%%%%%%%%%%%%%%%%%%%%%%%%%%%%%%%%%%%%%%%%%%%%%%%%%%%%%%%
\end{enumerate}
\fi
\begin{align}
&-\nabla \mathcal{L}_{q_k} (\bm{x}, \bm{\bar{\lambda}}_k) (\bm{y} (\bm{x},\alpha) - \bm{x}) 
\leq -\nabla_x \mathcal{L}_{q_k} (\bm{x}, \bm{\bar{\lambda}}_k) (\bm{y} (\bm{x},\beta) - \bm{x}), \ \text{for} \ 0 < \alpha \leq \beta, \label{thm:aug-lag-mult-error.1} \\
&E_{m,0} (\bm{x}, \bm{\bar{\lambda}}_k + 2 q_k \bm{h} (\bm{x}), \bm{\mu}(\bm{x},\beta)/\beta) 
\leq E_{m,0} (\bm{x}, \bm{\bar{\lambda}}_k + 2 q_k \bm{h} (\bm{x}), \bm{\mu}(\bm{x},\alpha)/\alpha), \ \text{for} \ 0 < \alpha \leq \beta \leq 1. \label{thm:aug-lag-mult-error.2}
\end{align}
\end{lemma}
\begin{proof}
Suppose $0 < \alpha \leq \beta$. Let $\bm{d}^{\alpha} (\bm{x}) = \bm{y} (\bm{x}, \alpha) - \bm{x}$. By P1 in \cite{Hager2006} with $\bm{x} \gets \bm{x} - \alpha \bm{g} (\bm{x})$ and $\bm{y} \gets \bm{y} (\bm{x}, \beta)$, we have that
\begin{align}
0 
&\leq (\bm{y} (\bm{x}, \alpha) - (\bm{x} - \alpha \bm{g} (\bm{x})))^{\intercal} (\bm{y} (\bm{x}, \beta) - \bm{y} (\bm{x},\alpha)) \\
&= ( \bm{d}^{\alpha} (\bm{x}) + \alpha \bm{g} (\bm{x}))^{\intercal} (\bm{d}^{\beta} (\bm{x}) - \bm{d}^{\alpha} (\bm{x})) \\
&= \bm{d}^{\alpha} (\bm{x})^{\intercal} \bm{d}^{\beta} (\bm{x}) - \| \bm{d}^{\alpha} (\bm{x}) \|^2 + \alpha \bm{g} (\bm{x})^{\intercal} \bm{d}^{\beta} (\bm{x}) - \alpha \bm{g} (\bm{x})^{\intercal} \bm{d}^{\alpha} (\bm{x}).
\end{align}
Rearranging terms and dividing both sides by $\alpha$ yields
\begin{align}
- \bm{g} (\bm{x})^{\intercal} \bm{d}^{\beta} (\bm{x})
&\leq - \bm{g} (\bm{x})^{\intercal} \bm{d}^{\alpha} (\bm{x}) + \frac{1}{\alpha} \bm{d}^{\alpha} (\bm{x})^{\intercal} \bm{d}^{\beta} (\bm{x}) - \frac{1}{\alpha} \| \bm{d}^{\alpha} (\bm{x}) \|^2.
\end{align}
It now follows from Cauchy-Schwarz that
\begin{align}
- \bm{g} (\bm{x})^{\intercal} \bm{d}^{\beta} (\bm{x})
&\leq - \bm{g} (\bm{x})^{\intercal} \bm{d}^{\alpha} (\bm{x}) + \frac{1}{\alpha} \| \bm{d}^{\alpha} (\bm{x}) \| \| \bm{d}^{\beta} (\bm{x}) \| - \frac{1}{\alpha} \| \bm{d}^{\alpha} (\bm{x}) \|^2. \label{eq:aug-lag-mult-error.1}
\end{align}
By swapping $\alpha$ and $\beta$ in the construction of (\ref{eq:aug-lag-mult-error.1}), we obtain the inequality
\begin{align}
- \bm{g} (\bm{x})^{\intercal} \bm{d}^{\alpha} (\bm{x})
&\leq - \bm{g} (\bm{x})^{\intercal} \bm{d}^{\beta} (\bm{x}) + \frac{1}{\beta} \| \bm{d}^{\beta} (\bm{x}) \| \| \bm{d}^{\alpha} (\bm{x}) \| - \frac{1}{\beta} \| \bm{d}^{\beta} (\bm{x}) \|^2. \label{eq:aug-lag-mult-error.1.1}
\end{align}

As $\alpha \leq \beta$, by P4 in \cite{Hager2006} we have that $\| \bm{d}^{\alpha} (\bm{x}) \| \leq \| \bm{d}^{\beta} (\bm{x}) \|$. Substituting this into (\ref{eq:aug-lag-mult-error.1.1}) yields
\begin{align}
- \bm{g} (\bm{x})^{\intercal} \bm{d}^{\alpha} (\bm{x})
&\leq - \bm{g} (\bm{x})^{\intercal} \bm{d}^{\beta} (\bm{x}) + \frac{1}{\beta} \| \bm{d}^{\beta} (\bm{x}) \|^2 - \frac{1}{\beta} \| \bm{d}^{\beta} (\bm{x}) \|^2 
= - \bm{g} (\bm{x})^{\intercal} \bm{d}^{\beta} (\bm{x}). \label{eq:aug-lag-mult-error.2.1}
\end{align}
which is equivalent to (\ref{thm:aug-lag-mult-error.1}). 

%%%%%%%%%%%%%%%% NOT NEEDED FOR PAPER %%%%%%%%%%%%%%%%%%%%%%
\iffalse 
By P5 in \cite{Hager2006}, we have that $\| \bm{d}^{\beta} (\bm{x}) \| \leq \frac{\beta}{\alpha} \| \bm{d}^{\alpha} (\bm{x}) \|$ which combined with (\ref{eq:aug-lag-mult-error.1}) yields
\begin{align}
- \bm{g} (\bm{x})^{\intercal} \bm{d}^{\beta} (\bm{x})
&\leq - \bm{g} (\bm{x})^{\intercal} \bm{d}^{\alpha} (\bm{x}) + \frac{\beta}{\alpha^2} \| \bm{d}^{\alpha} (\bm{x}) \|^2 - \frac{1}{\alpha} \| \bm{d}^{\alpha} (\bm{x}) \|^2 \nonumber \\
&= - \bm{g} (\bm{x})^{\intercal} \bm{d}^{\alpha} (\bm{x}) + \left( \frac{\beta - \alpha}{\alpha} \right) \frac{\| \bm{d}^{\alpha} (\bm{x}) \|^2}{\alpha}. \label{eq:aug-lag-mult-error.2.2}
\end{align}
By P6 in \cite{Hager2006} and the fact that $\beta - \alpha \geq 0$, we have that
\begin{align}
\left( \frac{\beta - \alpha}{\alpha} \right) \frac{\| \bm{d}^{\alpha} (\bm{x}) \|^2}{\alpha}
\leq \left( \frac{\beta - \alpha}{\alpha} \right) \left( - \bm{g} (\bm{x})^{\intercal} \bm{d}^{\alpha} (\bm{x}) \right)
= \left( 1 - \frac{\beta}{\alpha} \right) \bm{g} (\bm{x})^{\intercal} \bm{d}^{\alpha} (\bm{x}). \label{eq:aug-lag-mult-error.2.3}
\end{align}
Now (\ref{eq:aug-lag-mult-error.2.2}) and (\ref{eq:aug-lag-mult-error.2.3}) yield
\begin{align}
- \bm{g} (\bm{x})^{\intercal} \bm{d}^{\beta} (\bm{x}) 
&\leq - \frac{\beta}{\alpha}  \bm{g} (\bm{x})^{\intercal} \bm{d}^{\alpha} (\bm{x}) \label{eq:aug-lag-mult-error.2.4}
\end{align}
which is equivalent to (\ref{thm:aug-lag-mult-error.2}old).
\fi
%%%%%%%%%%%%%%%%%%%%%%%%%%%%%%%%%%%%%%%%%%%%%%%%%%%%%%%%%%%%

Now suppose that $\alpha \leq \beta \leq 1$. By P5 in \cite{Hager2006}, we have that $\| \bm{d}^{\beta} (\bm{x}) \| \leq \frac{\beta}{\alpha} \| \bm{d}^{\alpha} (\bm{x}) \|$. Substituting this into (\ref{eq:aug-lag-mult-error.1}) yields
\begin{align}
- \bm{g} (\bm{x})^{\intercal} \bm{d}^{\beta} (\bm{x})
&\leq - \bm{g} (\bm{x})^{\intercal} \bm{d}^{\alpha} (\bm{x}) + \frac{\beta}{\alpha^2} \| \bm{d}^{\alpha} (\bm{x}) \|^2 - \frac{1}{\alpha} \| \bm{d}^{\alpha} (\bm{x}) \|^2 \\
&= - \bm{g} (\bm{x})^{\intercal} \bm{d}^{\alpha} (\bm{x}) + \left( \frac{1}{\alpha^2} - \frac{1}{\alpha} \right) \| \bm{d}^{\alpha} (\bm{x}) \|^2 - \frac{1 - \beta}{\alpha^2} \| \bm{d}^{\alpha} (\bm{x}) \|^2. \label{eq:aug-lag-mult-error.3}
\end{align}
Again, by P5 in \cite{Hager2006} we have that $-\frac{1}{\alpha^2} \| \bm{d}^{\alpha} (\bm{x}) \|^2 \leq - \frac{1}{\beta^2} \| \bm{d}^{\beta} (\bm{x}) \|$ and since $1 - \beta \geq 0$ it follows that 
\begin{align}
- \frac{1 - \beta}{\alpha^2} \| \bm{d}^{\alpha} (\bm{x}) \|^2 
\leq - \frac{1 - \beta}{\beta^2} \| \bm{d}^{\beta} (\bm{x}) \|^2
= - \left( \frac{1}{\beta^2} - \frac{1}{\beta} \right) \| \bm{d}^{\beta} (\bm{x}) \|^2. \label{eq:aug-lag-mult-error.3.1}
\end{align}
Combining (\ref{eq:aug-lag-mult-error.3}) and (\ref{eq:aug-lag-mult-error.3.1}) and rearranging terms yields
\begin{align}
- \bm{g} (\bm{x})^{\intercal} \bm{d}^{\beta} (\bm{x}) + \left( \frac{1}{\beta^2} - \frac{1}{\beta} \right) \| \bm{d}^{\beta} (\bm{x}) \|^2
&\leq - \bm{g} (\bm{x})^{\intercal} \bm{d}^{\alpha} (\bm{x}) + \left( \frac{1}{\alpha^2} - \frac{1}{\alpha} \right) \| \bm{d}^{\alpha} (\bm{x}) \|^2 \label{eq:aug-lag-mult-error.3.2}
\end{align}
which by (\ref{eq:aug-lag-mod.4}) is equivalent to (\ref{thm:aug-lag-mult-error.2}).
%%%%%%%%%%%%%%%% NOT NEEDED FOR PAPER %%%%%%%%%%%%%%%%%%%%%%
\iffalse 
Now suppose that $1 \leq \alpha \leq \beta \leq 3$. By P5 in \cite{Hager2006}, we have that $-\frac{1}{\alpha^2} \| \bm{d}^{\alpha} (\bm{x}) \|^2 \leq - \frac{1}{\beta^2} \| \bm{d}^{\beta} (\bm{x}) \|$. Combining this with $\frac{1 - \beta}{\beta^3} \leq \frac{1 - \alpha}{\alpha^3} \leq 0$ for $1 \leq \alpha \leq \beta \leq 3$ yields
\begin{align}
\alpha \left(\frac{1 - \beta}{\beta^2} \right) \| \bm{d}^{\beta} (\bm{x}) \|^2 
\leq \beta \left( \frac{1 - \alpha}{\alpha^2} \right) \| \bm{d}^{\alpha} (\bm{x}) \|^2.
\label{eq:aug-lag-mult-error.4}
\end{align}
As $\alpha \leq \beta$, combining (\ref{thm:aug-lag-mult-error.2}) with (\ref{eq:aug-lag-mult-error.4}) it follows that
\begin{align}
- \alpha \bm{g} (\bm{x})^{\intercal} \bm{d}^{\beta} (\bm{x}) + \alpha \left( \frac{1}{\beta^2} - \frac{1}{\beta} \right) \| \bm{d}^{\beta} (\bm{x}) \|^2
&\leq - \beta \bm{g} (\bm{x})^{\intercal} \bm{d}^{\alpha} (\bm{x}) + \beta \left( \frac{1}{\alpha^2} - \frac{1}{\alpha} \right) \| \bm{d}^{\alpha} (\bm{x}) \|^2 \label{eq:aug-lag-mult-error.4.1}
\end{align}
By the formula for $E_{m,0}$ in (\ref{eq:aug-lag-mod.4}), we have
\begin{align}
\alpha E_{m,0} (\bm{x}, \bm{\bar{\lambda}}_k + 2 q \bm{h} (\bm{x}), \bm{\mu}(\bm{x},\beta)/\beta) 
\leq \beta E_{m,0} (\bm{x}, \bm{\bar{\lambda}}_k + 2 q \bm{h} (\bm{x}), \bm{\mu}(\bm{x},\alpha)/\alpha) \label{eq:aug-lag-mult-error.4.2}
\end{align}
which shows that (\ref{thm:aug-lag-mult-error}.4) holds.
\fi
%%%%%%%%%%%%%%%%%%%%%%%%%%%%%%%%%%%%%%%%%%%%%%%%%%%%%%%%%%%%
\end{proof}

Before providing a convergence result for this approach to phase one of NPASA, we consider how convergence of Algorithm~\ref{alg:pasa-aug-lag} for the minimization problem in the GS algorithm relates to the multiplier error estimator of the original nonlinear program, (\ref{prob:main-nlp}). In particular, since $E_{m,1} (\bm{x}_k, \bm{\lambda}_k, \bm{\mu}_k) \leq E_{m,0} (\bm{x}_k, \bm{\lambda}_k, \bm{\mu}_k)$ and our branching criterion in phase one of NPASA is given by $E_{m,1} (\bm{x}_k, \bm{\lambda}_k, \bm{\mu}_k) \leq \theta E_c (\bm{x}_k)$, we would like some conditions under which Algorithm~\ref{alg:pasa-aug-lag} converges to a point $\bm{x}_k$ at which $E_{m,0} (\bm{x}_k, \bm{\lambda}_k, \bm{\mu}_k) \approx 0$. Before providing such a result,
%The following result is a natural first step as it is a global convergence result for the multiplier error based on the global convergence result for PASA in \cite{Hager2016} that only requires one assumption in addition to those already imposed in \cite{Hager2016} to establish global convergence of PASA. First, 
we need to set up the assumptions that are used in this result as they are based on details of PASA \cite{HagerActive}. 

Given $\bm{u} \in \mathbb{R}^n$, $\bm{\nu} \in \mathbb{R}^{\ell}$, and $q \in \mathbb{R}$ we define the level set
\begin{align}
\mathcal{S}(\bm{u}, \bm{\nu}, q) := \{ \bm{x} \in \Omega : \mathcal{L}_q (\bm{x}, \bm{\nu}) \leq \mathcal{L}_q (\bm{u}, \bm{\nu}) \}. \label{def:level-set-S}
\end{align}
Additionally, we let the set $\mathcal{D}(\bm{u}, \bm{\nu}, q)$ denote the search directions generated by phase one of the algorithm PASA \cite{HagerActive} when solving the problem 
\begin{align}
\begin{array}{cc}
\displaystyle \min_{\bm{x}} & \displaystyle \mathcal{L}_q (\bm{x}, \bm{\nu}) \\
\text{s.t.} & \bm{x} \in \Omega
\end{array} \label{prob:aug-lag-intro}
\end{align}
starting at the initial guess $\bm{x} = \bm{u}$. Note that the definition of the search directions generated in phase one of PASA can be found in Algorithm 1 in Section 2 of \cite{HagerActive}. We are now ready to state the assumptions:
\begin{enumerate}[leftmargin=2\parindent,align=left,labelwidth=\parindent,labelsep=7pt]
\item[(G1)] Given $\bm{u} \in \mathbb{R}^n$ and $\bm{\nu} \in \mathbb{R}^{\ell}$, $\mathcal{L}_q (\bm{x}, \bm{\nu})$ is bounded from below on the level set $\mathcal{S} (\bm{u}, \bm{\nu})$ and $d_{\max} = \sup_{\bm{d} \in \mathcal{D}(\bm{u}, \bm{\nu}, q)} \| \bm{d} \| < \infty$.
\item[(G2)] Given $\bm{u} \in \mathbb{R}^n$ and $\bm{\nu} \in \mathbb{R}^{\ell}$, if $\overline{\mathcal{S}}(\bm{u}, \bm{\nu})$ is the collection of $\bm{x} \in \Omega$ whose distance to $\mathcal{S}(\bm{u}, \bm{\nu})$ is at most $d_{\max}$, then $\nabla_x \mathcal{L}_q (\bm{x}, \bm{\nu})$ is Lipschitz continuous on $\overline{\mathcal{S}}(\bm{u}, \bm{\nu})$.
\end{enumerate}

With Lemma~\ref{thm:aug-lag-mult-error} established and (G1) and (G2) stated, we are now ready to provide a convergence result for the nonmonotone gradient projection algorithm (NGPA) in phase one of PASA. %While some aspects of the proof are similar to Theorem 2.2 in \cite{Hager2006}, we provide most of the details here for completion.
%This convergence theorem established for the nonmonotone gradient projection algorithm (NGPA) in \cite{Hager2006} 
that is necessary to establish convergence when using Algorithm~\ref{alg:pasa-aug-lag} to solve 
%the augmented Lagrangian 
problem (\ref{prob:aug-lag-mod}).

\begin{theorem} \label{thm:ngpa-conv}
Given $\bm{u}_0 \in \mathbb{R}^n$, $\bm{\bar{\lambda}}_k$, and $q_k$ suppose that (G1) and (G2) hold at $(\bm{u}_0, \bm{\bar{\lambda}}_k, q_k)$. If the nonmonotone gradient projection algorithm (NGPA) in \cite{Hager2006} with $\varepsilon = 0$ is used to solve problem (\ref{prob:aug-lag-mod}) then NGPA either terminates in a finite number of iterations at a stationary point, or generates a sequence of iterates $\{ \bm{u}_i \}_{i=0}^{\infty}$ satisfying
\begin{align}
\liminf_{i \to \infty} E_{m,0} (\bm{u}_i, \bm{\bar{\lambda}}_k + 2 q_k \bm{h} (\bm{u}_i), \bm{\mu}(\bm{u}_i,1)) = 0.
\end{align}
\end{theorem}
\begin{proof}
Let $F(\bm{x}) = \mathcal{L}_{q_k} (\bm{x}, \bm{\bar{\lambda}}_k)$. First, note that from (3.6) in \cite{Hager2016}, the stepsize $s_k$ in the Armijo line search is bounded below by
%\begin{align}
$s_k \geq \min \left\{ 1, \frac{2 \eta (1 - \delta)}{\kappa \alpha_{\max}} \right\} =: c$,
%\end{align}
where $\kappa$ is the Lipschitz constant for $\nabla_x \mathcal{L}_{q_k} (\bm{x}, \bm{\bar{\lambda}}_k)$. From equation (2.16) in \cite{Hager2006} we have that
\begin{align}
- \delta c \bm{g} (\bm{u}_i)^{\intercal} (\bm{y} (\bm{u}_i, \bar{\alpha}_i) - \bm{u}_i) \leq F_i^r - F (\bm{u}_{i+1}), \label{eq:gpa-conv.1}
\end{align}
for some $\bar{\alpha}_i \in [\alpha_{\min}, \alpha_{\max}]$. We now proceed by way of contradiction. Accordingly, suppose that 
\begin{align}
\liminf_{i \to \infty} E_{m,0} (\bm{u}_i, \bm{\bar{\lambda}}_k + 2 q_k \bm{h} (\bm{u}_i), \bm{\mu}(\bm{u}_i,1)) > 0.
\end{align}
Then there exists a constant $\tau > 0$ such that 
\begin{align}
\tau \leq E_{m,0} (\bm{u}_i, \bm{\bar{\lambda}}_k + 2 q_k \bm{h} (\bm{u}_i), \bm{\mu}(\bm{u}_i,1)) 
= - \bm{g} (\bm{u}_i)^{\intercal} (\bm{y} (\bm{u}_i, 1) - \bm{u}_i), \label{eq:gpa-conv.2}
\end{align}
for all $i \geq 0$. Suppose $\bar{\alpha}_i \geq 1$. By (\ref{thm:aug-lag-mult-error.1}) in Theorem~\ref{thm:aug-lag-mult-error} with $\alpha = 1$ and $\beta = \bar{\alpha}_i$, we have
\begin{align}
\tau \leq - \bm{g} (\bm{u}_i)^{\intercal} (\bm{y} (\bm{u}_i, 1) - \bm{u}_i) 
\leq - \bm{g} (\bm{u}_i)^{\intercal} (\bm{y} (\bm{u}_i, \bar{\alpha}_i) - \bm{u}_i). \label{eq:gpa-conv.3}
\end{align}
On the other hand, if $\bar{\alpha}_i \in (0,1)$ then by (\ref{thm:aug-lag-mult-error.2}) in Theorem~\ref{thm:aug-lag-mult-error} with $\alpha = \bar{\alpha}_i$ and $\beta = 1$ we have
\begin{align}
- \bm{g} (\bm{u}_i)^{\intercal} (\bm{y} (\bm{u}_i, 1) - \bm{u}_i) 
\leq - \bm{g} (\bm{u}_i)^{\intercal} (\bm{y} (\bm{u}_i, \bar{\alpha}_i) - \bm{u}_i) + \left( \frac{1}{\bar{\alpha}_i} - 1 \right) \frac{\| \bm{y} (\bm{u}_i, \bar{\alpha}_i) - \bm{u}_i \|^2}{\bar{\alpha}_i}. \label{eq:gpa-conv.4}
\end{align}
From P6 in \cite{Hager2006} we have that 
\begin{align}
\frac{\| \bm{y} (\bm{u}_i, \bar{\alpha}_i) - \bm{u}_i \|^2}{\bar{\alpha}_i} \leq - \bm{g} (\bm{u}_i)^{\intercal} (\bm{y} (\bm{u}_i, \bar{\alpha}_i). \label{eq:gpa-conv.5}
\end{align}
As $\bar{\alpha}_i \in (0,1)$, we have that $\left( \frac{1}{\bar{\alpha}_i} - 1 \right) > 0$ which combined with (\ref{eq:gpa-conv.5}) yields
\begin{align}
\left( \frac{1}{\bar{\alpha}_i} - 1 \right) \frac{\| \bm{y} (\bm{u}_i, \bar{\alpha}_i) - \bm{u}_i \|^2}{\bar{\alpha}_i} 
&\leq \left( \frac{1}{\bar{\alpha}_i} - 1 \right) \left( - \bm{g} (\bm{u}_i)^{\intercal} (\bm{y} (\bm{u}_i, \bar{\alpha}_i) \right) %\nonumber \\
= \left( 1 - \frac{1}{\bar{\alpha}_i} \right) \bm{g} (\bm{u}_i)^{\intercal} (\bm{y} (\bm{u}_i, \bar{\alpha}_i). \label{eq:gpa-conv.6}
\end{align}
From (\ref{eq:gpa-conv.4}) and (\ref{eq:gpa-conv.6}) it now follows that
\begin{align}
- \bm{g} (\bm{u}_i)^{\intercal} (\bm{y} (\bm{u}_i, 1) - \bm{u}_i) \leq - \frac{1}{\bar{\alpha}_i} \bm{g} (\bm{u}_i)^{\intercal} (\bm{y} (\bm{u}_i, \bar{\alpha}_i) - \bm{u}_i). \label{eq:gpa-conv.7}
\end{align}
As $\bar{\alpha}_i \geq \alpha_{\min}$, (\ref{eq:gpa-conv.2}) and (\ref{eq:gpa-conv.7}) now yield
\begin{align}
\alpha_{\min} \tau \leq - \bm{g} (\bm{u}_i)^{\intercal} (\bm{y} (\bm{u}_i, \bar{\alpha}_i) - \bm{u}_i). \label{eq:gpa-conv.8}
\end{align}
Hence, by observing (\ref{eq:gpa-conv.3}) and (\ref{eq:gpa-conv.8}) we have established that for any choice of $\bar{\alpha}_i \in [\alpha_{\min}, \alpha_{\max}]$ there exists a constant $\nu$ such that
\begin{align}
\nu \leq - \bm{g} (\bm{u}_i)^{\intercal} (\bm{y} (\bm{u}_i, \bar{\alpha}_i) - \bm{u}_i). \label{eq:gpa-conv.9}
\end{align}
Combining (\ref{eq:gpa-conv.1}) and (\ref{eq:gpa-conv.9}) yields
%\begin{align}
$F (\bm{u}_{i+1}) \leq F_i^r - \delta c \nu$ 
%\label{eq:gpa-conv.10}
%\end{align}
which, by the same argument in the proof of Theorem 2.2 \cite{Hager2006}, results in a contradiction. Thus, 
%we have
%\begin{align}
$\liminf_{i \to \infty} E_{m,0} (\bm{u}_i, \bm{\bar{\lambda}}_k + 2 q_k \bm{h} (\bm{u}_i), \bm{\mu}(\bm{u}_i,1)) = 0$.
%\end{align}
%which concludes the proof.
\end{proof}

We now establish a global convergence result for Algorithm~\ref{alg:pasa-aug-lag} when solving this problem. We note that this is similar to Theorem 3.2 in \cite{Hager2016}. 

%%%%%%%%%%%%%%%%%%%%%%%%%%%%%%%

\begin{theorem} \label{thm:pasa-aug-lag-mod-global-conv}
Given $\bm{u}_0 \in \mathbb{R}^n$, $\bm{\bar{\lambda}}_k$, and $q_k$ suppose that (G1) and (G2) hold at $(\bm{u}_0, \bm{\bar{\lambda}}_k, q_k)$. If Algorithm~\ref{alg:pasa-aug-lag} with $\varepsilon = 0$ is used to solve problem (\ref{prob:aug-lag-mod}) then Algorithm~\ref{alg:pasa-aug-lag} either terminates in a finite number of iterations at a stationary point, or generates a sequence of iterates $\{ \bm{u}_i \}_{i=0}^{\infty}$ satisfying
\begin{align}
\liminf_{i \to \infty} E_{m,0} (\bm{u}_i, \bm{\bar{\lambda}}_k + 2 q_k \bm{h} (\bm{u}_i), \bm{\mu}(\bm{u}_i,1)) = 0. \label{result:pasa-aug-lag-mod-global-conv}
\end{align}
\end{theorem}
\begin{proof}
The proof is split into three cases.\\
{\bfseries Case 1: Phase two of Algorithm~\ref{alg:pasa-aug-lag} is executed a finite number of times.} In this case, only phase one of Algorithm~\ref{alg:pasa-aug-lag} is executed for $i$ sufficiently large. As such, from Theorem~\ref{thm:ngpa-conv} it follows that (\ref{result:pasa-aug-lag-mod-global-conv}) holds. \\
{\bfseries Case 2: Phase one of Algorithm~\ref{alg:pasa-aug-lag} is executed a finite number of times.} Here, only phase two is executed for $i$ sufficiently large. Since phase two eventually does not branch back to phase one, by the branching criterion in phase two we have that
\begin{align}
\theta E_{m,0} (\bm{u}_i, \bm{\bar{\lambda}}_k + 2 q_k \bm{h} (\bm{u}_i), \bm{\mu}(\bm{u}_i,1)) \leq e_{\scaleto{PASA}{4pt}} (\bm{u}_i)
\end{align}
for $i$ sufficiently large. Using an argument found in the proof of Theorem 3.2 from \cite{Hager2016} we have that (\ref{result:pasa-aug-lag-mod-global-conv}) holds. \\
{\bfseries Case 3: Phases one and two of Algorithm~\ref{alg:pasa-aug-lag} are executed an infinite number of times.} By way of contradiction, assume that (\ref{result:pasa-aug-lag-mod-global-conv}) does not hold. Then there exists a constant $\tau > 0$ such that 
\begin{align}
\tau \leq E_{m,0} (\bm{u}_i, \bm{\bar{\lambda}}_k + 2 q_k \bm{h} (\bm{u}_i), \bm{\mu}(\bm{u}_i,1)) = - \bm{g} (\bm{u}_i)^{\intercal} (\bm{y} (\bm{u}_i, 1) - \bm{u}_i).
\end{align}
Letting $F (\bm{u}_{i+1}) = \mathcal{L}_{q_k} (\bm{u}_{i+1}, \bm{\bar{\lambda}}_k)$ and applying the same argument as in the proof of Theorem~\ref{thm:ngpa-conv}, we have that there exists a constant $\nu$ such that
\begin{align}
F (\bm{u}_{i+1}) \leq F (\bm{u}_i) - \delta c \nu \label{eq:pasa-aug-lag-mod-global-conv.10}
\end{align}
for each iteration of GPA when $i$ is sufficiently large. Additionally, by property F1 of the LCO in \cite{Hager2016}, we have that $F (\bm{u}_{i+1}) \leq F (\bm{u}_i)$. Since there are an infinite number of iterations in phase one, it follows that the cost function is unbounded from below which contradictions assumption (G1). Thus, (\ref{result:pasa-aug-lag-mod-global-conv}) holds.
\end{proof}

Based on the analysis in this section, Algorithm~\ref{alg:pasa-aug-lag} should be used to solve the minimization problem in Algorithm~\ref{alg:gs}, the GS algorithm. 
%This follows from comparing the assumptions required for convergence of the multiplier error estimator for the original nonlinear program when using PASA or Algorithm~\ref{alg:pasa-aug-lag} to solve problem (\ref{prob:aug-lag-mod}) in Theorem~\ref{thm:pasa-aug-lag} and \ref{thm:pasa-aug-lag-mod-global-conv}. While these Theorems yield the same result, more assumptions are used to achieve this result when solving problem (\ref{prob:aug-lag-mod}) with PASA.
To conclude this section, we note that the inequality multiplier $\bm{\mu} (\bm{x}_k, 1)$ required to check the convergence criterion in phase one of NPASA can be explicitly constructed using values computed during the solution of the GS algorithm. Details of this construction are provided in the companion paper \cite{diffenderfer2020local}.

\section{Convergence Analysis for Phase Two of NPASA} \label{section:npasa-phase-two}
In this section, we provide conditions under which the constraint and multiplier steps in the LS algorithm satisfy desirable convergence properties. We first provide a result on the perturbed Newton step scheme used in the constraint step in Section~\ref{subsec:infeas}. Then we establish a convergence result for the constraint step in Section~\ref{subsec:constraint-step-analysis} followed by a result for the multiplier step of NPASA in Section~\ref{subsec:multiplier-step-analysis}. 
\subsection{Algorithm with Infeasibility Detection for Solving Constraint Step} \label{subsec:infeas}
One method for minimizing the violation of the equality constraints in problem (\ref{prob:main-nlp}) is to generate a sequence of iterates $\{ \bm{w}_i \}_{i=1}^{\infty}$ defined by $\bm{w}_{i+1} = \bm{w}_i + \bm{d}_i$, where
%\begin{align}
$\bm{d}_i$ 
is the solutioin to the minimization problem given by  $\min \left\{ \| \bm{w} - \bm{d}_{i-1} \|^2 : \nabla \bm{h} (\bm{d}_{i-1}) (\bm{w} - \bm{d}_{i-1}) = - \bm{h} (\bm{d}_{i-1}), \bm{w} \in \Omega \right\}$.
%\end{align}
% Which is applying Newton's method to the equality constraints with a unit stepsize
Given $\bar{\bm{x}}$, PPROJ \cite{Hager2016} %, developed by Hager and Zhang, 
uses a dual method for solving the projection problem
\begin{align}
\begin{array}{cc}
\displaystyle \min_{\bm{x}} & \frac{1}{2} \| \bar{\bm{x}} - \bm{x} \|^2 \\
\text{s.t.} & \bm{l} \leq \bm{A} \bm{x} \leq \bm{u}, \ \ \bm{x} \geq \bm{0}.
\end{array}
\end{align}
then reconstructs the primal solution. We would like to use this algorithm to solve for the Newton direction, $\bm{d}_i$. However, since $\{ \bm{w} \in \mathbb{R}^n : \bm{h} (\bm{w}) = \bm{0}, \bm{w} \in \Omega \} \neq \emptyset$ does not guarantee
\begin{align}
\left\{ \bm{w} \in \mathbb{R}^n : \nabla \bm{h} (\bm{d}_i) (\bm{w} - \bm{d}_i) = - \bm{h} (\bm{d}_i), \bm{w} \in \Omega \right\} \neq \emptyset,
\end{align} 
it may be necessary to perturb the constraint $\nabla \bm{h} (\bm{d}_i) (\bm{w} - \bm{d}_i) = - \bm{h} (\bm{d}_i)$ in order to solve for $\bm{d}_{i+1}$.

In order to simplify the notation in our discussion of this problem, we express the linearized constraint set as
%\begin{align}
$\left\{ \bm{w} \in \mathbb{R}^n : \bm{M} (\bm{w} - \bar{\bm{w}}) = \bm{c}, \bm{w} \geq \bm{0} \right\}$
%\end{align}
supposing that we are given $\bm{M} \in \mathbb{R}^{m \times n}$, $\bm{c} \in \mathbb{R}^m$, and $\bar{\bm{w}} \in \mathbb{R}^n$. Ideally, we would like to determine the solution to the problem 
\begin{align}
&\min_{\bm{w} \in \mathcal{S}} \ \frac{1}{2} \| \bm{w} - \bar{\bm{w}} \|^2, \label{Sprob}
\end{align}
%\begin{align}
%&\min_{x \geq 0} \ \frac{1}{2} \| b - A(x - \bar{x}) \|^2, \label{FeasibilityProb}
%\end{align}
%%%given $A$, $b$, and $\bar{x}$. 
%Ideally, if we knew a solution to problem (\ref{FeasibilityProb}) then we would know whether or not $\mathcal{F}_0 = \emptyset$ since $\mathcal{F}_0 = \emptyset$ if and only if $\| b - A(x - \bar{x}) \| > 0$ for all $x \geq 0$. 
where $\mathcal{S} = \arg\min_{\bm{w} \geq \bm{0}} \left\{ \| \bm{c} - \bm{M} (\bm{w} - \bar{\bm{w}}) \|^2 \right\}$. Since problem (\ref{Sprob}) involves minimizing a strongly convex objective function over a convex set, there exists a unique solution which we will denote by $\bm{w}_{\infty}$. As such, every search direction from $\bm{w}_{\infty}$ leading to a point in $\mathcal{S}$ must be an ascent direction. Hence, the first order optimality conditions for $\bm{w}_{\infty}$ can be expressed as 
\begin{align}
\langle \bm{w}_{\infty} - \bar{\bm{w}}, \bm{w} - \bm{w}_{\infty} \rangle \geq 0, \ \text{for all} \ \bm{w} \in \mathcal{S}. 
\label{problemS}
\end{align}
Since determining elements of the constraint set for problem (\ref{Sprob}) requires knowing all minimizers of $\| \bm{c} - \bm{M}(\bm{w} - \bar{\bm{w}}) \|$ we consider an alternative problem. 

To simplify this problem, we define the sets 
%\begin{align}
$\mathcal{F} = \left\{ 
[\bm{w}, \bm{y}]^{\intercal} \in \mathbb{R}^{n+m} : \bm{M}(\bm{w} - \bar{\bm{w}}) + \bm{y} = \bm{c}, \bm{w} \geq \bm{0} \right\}$.
%\end{align}
Now, given a penalty parameter $p > 0$, the problem is defined by
\begin{align}
\min_{[\bm{w}, \bm{y}]^{\intercal} \in \mathcal{F}} \ \ f_p(\bm{w},\bm{y}) := \frac{1}{2} \| \bm{w} - \bar{\bm{w}} \|^2 + \frac{p}{2} \| \bm{y} \|^2. \label{prob:Pp}
\end{align}
Since problem (\ref{prob:Pp}) involves minimizing a strongly convex objective function over a convex set, it has a unique solution which we will denote by $[ \bm{w}_p, \bm{y}_p]^{\intercal}$. Note that the first order optimality conditions for $[ \bm{w}_p, \bm{y}_p ]^{\intercal}$ can be expressed as
\begin{align}
\langle \bm{w}_p - \bar{\bm{w}}, \bm{w} - \bm{w}_p \rangle + p \langle \bm{y}_p, \bm{y} - \bm{y}_p \rangle \geq 0, \ \text{for all} \ [\bm{w}, \bm{y}]^{\intercal} \in \mathcal{F}. \label{POptCond}
\end{align}
While the number of primal variables in (\ref{prob:Pp}) is greater than that of problem (\ref{Sprob}), the dimension of the dual variables is the same for problems (\ref{Sprob}) and (\ref{prob:Pp}). Hence, by using the dual method PPROJ \cite{Hager2016} we can solve problem (\ref{prob:Pp}) without increasing the dimension of the optimization problem.
%%%That is, the complexity is the same for solving (P) and (\ref{prob:Pp}) using PPROJ \cite{Hager2016}. 
The following result provides valuable insight on the relationship between $\bm{w}_p$ and $\bm{w}_{\infty}$. Note that we make use of the set
%\begin{align}
$\mathcal{F}_{\delta} = \left\{ [\bm{w}, \bm{y}]^{\intercal} \in \mathbb{R}^{n+m} : \bm{M}(\bm{w} - \bar{\bm{w}}) + \bm{y} = \bm{c}, \bm{w} \geq \bm{0}, \| \bm{y} \|^2 \leq \delta \right\}$.
%\end{align}
in the proof of the following result.

\begin{theorem} \label{thm:feasibility-detection}
Let $[\bm{w}_p, \bm{y}_p]^{\intercal}$ be the unique solution to problem \emph{(\ref{prob:Pp})}. Then 
%\begin{align}
$\| \bm{w}_p - \bm{x}_{\infty} \| 
= O \left( \frac{1}{p} \right)$,
%\end{align}
for all $p > 0$ sufficiently large.
%%%where $x_{\infty} \in \mathcal{S}$ is the closest solution to $\bar{x}$.	
\end{theorem}
\begin{proof}
Letting $\gamma = \min_{\bm{w} \geq \bm{0}} \| \bm{c} - \bm{M}(\bm{w} - \bar{\bm{w}}) \|^2$ it follows that $\mathcal{S} = \{ \bm{w} \in \mathbb{R}^n : [\bm{w}, \bm{y}]^{\intercal} \in \mathcal{F}_{\gamma} \}$. So $\mathcal{S} \subseteq \mathcal{F}$ and by defining $\bm{y}_{\infty} = \bm{c} - \bm{M}(\bm{w}_{\infty} - \bar{\bm{w}})$ we have that $[\bm{w}_{\infty}, \bm{y}_{\infty}]^{\intercal} \in \mathcal{F}$. Since $[\bm{w}_p, \bm{y}_p]^{\intercal}$ minimizes the objective in problem (\ref{prob:Pp}), we now have that
\begin{align}
\frac{1}{2} \| \bm{w}_p - \bar{\bm{w}} \|^2 + \frac{p}{2} \| \bm{y}_p \|^2 
&\leq \frac{1}{2} \| \bm{w}_{\infty} - \bar{\bm{w}} \|^2 + \frac{p}{2} \| \bm{y}_{\infty} \|^2 %\nonumber \\
= \frac{1}{2} \| \bm{w}_{\infty} - \bar{\bm{w}} \|^2 + \frac{p}{2} \gamma. \label{objBound1}
\end{align}
Using (\ref{objBound1}) we observe that
%\begin{align}
$\| \bm{w}_p - \bar{\bm{w}} \|^2 
\leq \| \bm{w}_{\infty} - \bar{\bm{w}} \|^2 + p \left( \gamma - \| \bm{y}_p \|^2 \right)
\leq \| \bm{w}_{\infty} - \bar{\bm{w}} \|^2$,
%\end{align}
where the final inequality holds since $\| \bm{y}_p \|^2 \geq \gamma$, for all $p \geq 1$. Hence,
\begin{align}
\| \bm{w}_p \| \leq \| \bm{w}_{\infty} - \bar{\bm{w}} \| + \| \bar{\bm{w}} \|. \label{xpepBound}
\end{align}
Also, from (\ref{objBound1}) it follows that
%\begin{align}
$\| \bm{y}_p \|^2 \leq \frac{1}{p} \| \bm{w}_{\infty} - \bar{\bm{w}} \|^2 + \gamma$ %\label{yinfBound.0}
%\end{align}
which yields that
%\begin{align}
$\lim_{p \to \infty} \| \bm{y}_p \|^2 = \gamma$. %\label{yinfBound}
%\end{align} 
Since $\| \bm{w}_p \|$ is bounded for all values of $p$ in (\ref{xpepBound}), there exist convergent subsequences of $\{ \bm{w}_p \}_{p = 1}^{\infty}$. Letting $\{ \bm{w}_{p_k} \}_{k = 1}^{\infty}$ be a convergent subsequence of $\{ \bm{w}_p \}_{p = 1}^{\infty}$ with limit $\tilde{\bm{w}}$ we observe that 
\begin{align}
\| \bm{c} - \bm{M}(\tilde{\bm{w}} - \bar{\bm{w}}) \|^2 
&= \lim_{k \to \infty} \| \bm{c} - \bm{M}(\bm{w}_{p_k} - \bar{\bm{w}}) \|^2 
= \lim_{k \to \infty} \| \bm{y}_{p_k} \|^2 = \gamma.
\end{align}
So, taking $\tilde{\bm{y}} = \bm{c} - \bm{M}(\tilde{\bm{w}} - \bar{\bm{w}})$, we have
%\begin{align}
$\lim_{k \to \infty} [ \bm{w}_{p_k}, \bm{y}_{p_k} ]^{\intercal} 
= [ \tilde{\bm{w}}, \tilde{\bm{y}} ]^{\intercal} \in \mathcal{F}_{\gamma}$.
%\end{align}
Hence, $\tilde{\bm{w}} \in \mathcal{S}$. 

We now claim that $\tilde{\bm{w}} = \bm{w}_{\infty}$. Observing the optimality conditions in (\ref{problemS}) it follows that we need to show that $\langle \tilde{\bm{w}} - \bar{\bm{w}}, \bm{w} - \tilde{\bm{w}} \rangle \geq 0$, for all $\bm{w} \in \mathcal{S}$. Let $p \geq 1$ be given. By the optimality conditions in (\ref{POptCond}), we have that 
\begin{align}
\langle \bm{w}_p - \bar{\bm{w}}, \bm{w} - \bm{w}_p \rangle 
&\geq p \langle \bm{y}_p, \bm{y}_p - \bm{y} \rangle, \label{tildeOpt} 
\end{align}
for all $[\bm{w}, \bm{y}]^{\intercal} \in \mathcal{F}$. As $\mathcal{S} = \{ \bm{w} \in \mathbb{R}^n : [\bm{w}, \bm{y}]^{\intercal} \in \mathcal{F}_{\gamma} \}$ and $\mathcal{F}_{\gamma} \subseteq \mathcal{F}$, it suffices to show that $\langle \bm{y}_p, \bm{y}_p - \bm{y} \rangle \geq 0$, for all $\bm{y} \in \mathbb{R}^m$ with $\| \bm{y} \|^2 = \gamma$. Accordingly, let $\bm{y} \in \mathbb{R}^m$ with $\| \bm{y} \|^2 = \gamma$ be given and observe that 
\begin{align}
\langle \bm{y}_p, \bm{y}_p - \bm{y} \rangle 
%&= \langle \bm{y}_p, \bm{y}_p \rangle - \langle \bm{y}_p, \bm{y} \rangle \\
&= \| \bm{y}_p \|^2 - \frac{1}{2} \left( \| \bm{y}_p \|^2 - \| \bm{y}_p - \bm{y} \|^2 + \| \bm{y} \|^2 \right) 
= \frac{1}{2} \left( \| \bm{y}_p \|^2 + \| \bm{y}_p - \bm{y} \|^2 - \gamma \right). \label{ypBound2}
\end{align}
By our choice of $\gamma$ we have that $\| \bm{y}_p \|^2 \geq \gamma$. This fact together with the inequality in (\ref{ypBound2}) implies that
%\begin{align}
$\langle \bm{y}_p, \bm{y}_p - \bm{y} \rangle 
\geq \frac{1}{2} \| \bm{y}_p - \bm{y} \|^2 \geq 0$. %\label{ypBound3}
%\end{align}
Substituting 
this %(\ref{ypBound3}) 
into (\ref{tildeOpt}) yields that $\langle \bm{w}_p - \bar{\bm{w}}, \bm{w} - \bm{w}_p \rangle \geq 0$, for all $\bm{w} \in \mathcal{S}$. Thus,
%\begin{align}
$\langle \tilde{\bm{w}} - \bar{\bm{w}}, \bm{w} - \tilde{\bm{w}} \rangle 
= \lim_{k \to \infty} \langle \bm{w}_{p_k} - \bar{\bm{w}}, \bm{w} - \bm{w}_{p_k} \rangle 
\geq 0$, 
%\end{align}
for all $\bm{w} \in \mathcal{S}$. So we conclude that $\tilde{\bm{w}} = \bm{w}_{\infty}$.

Next we claim that the entire sequence $\{ \bm{w}_p \}_{p = 1}^{\infty}$ converges to $\bm{w}_{\infty}$. By way of contradiction, suppose that there exists a subsequence $\{ \bm{w}_{p_k} \}_{k = 1}^{\infty}$ and a $\delta > 0$ such that $\| \bm{w}_{p_k} - \bm{w}_{\infty} \| \geq \delta$ for all $k \geq 1$. Since $\{ \bm{w}_{p_k} \}_{k = 1}^{\infty}$ is bounded it has a convergent subsequence, say $\{ \bm{w}_{q_k} \}_{k = 1}^{\infty}$. As $\{ \bm{w}_{q_k} \}_{k = 1}^{\infty}$ is a subsequence of $\{ \bm{w}_p \}_{p = 1}^{\infty}$ and every convergent subsequence of $\{ \bm{w}_p \}_{p = 1}^{\infty}$ converges to $\bm{w}_{\infty}$, there exists a $M \in \mathbb{N}$ such that $\| \bm{w}_{q_k} - \bm{w}_{\infty} \| < \delta$ for all $k \geq M$, a contradiction. Thus, every subsequence of $\{ \bm{w}_{p} \}_{p = 1}^{\infty}$ converges to $\bm{w}_{\infty}$ implying that $\bm{w}_{p} \to \bm{w}_{\infty}$ as $p \to \infty$.

Lastly, it remains to show that $\| \bm{w}_p - \bm{w}_{\infty} \| = O(1/p)$. Accordingly, let $p \geq 1$ be given, let $\mathcal{B} := \{ i : \left( \bm{w}_p \right)_i = 0 \text{ and } 1 \leq i \leq n \}$ denote the active set at point $\bm{w}_p$, and let $\mathcal{I} := \{ 1, 2, \ldots, n \} \setminus \mathcal{B}$. We now consider the problem
\begin{align}
\min_{\bm{w}_{\mathcal{B}} = 0} \ f (\bm{w}) := \frac{1}{2} \| \bm{w} - \bar{\bm{w}} \|^2 + \frac{p}{2} \| \bm{c} + \bm{M} \bar{\bm{w}} - \bm{M} \bm{w} \|^2 \label{BindingProb}
\end{align}
where $\bm{w}_{\mathcal{B}}$ is the vector with components $w_i$, for all $i \in \mathcal{B}$. Letting $\bm{u}^*$ denote the solution to problem (\ref{BindingProb}), we claim that $\bm{u}^* = \bm{w}_p$. In order to justify this claim, first note that since $\bm{w}_p$ satisfies the constraint for problem (\ref{BindingProb}) we have that
\begin{align}
f(\bm{u}^*) 
&\leq f(\bm{w}_p) 
= f_p (\bm{w}_p, \bm{y}_p). \label{BindingProbIneq}
\end{align}
%%%\begin{align}
%%%\frac{1}{2} \| u^* - \bar{x} \|^2 + \frac{p}{2} \| b - A(u^* - \bar{x}) \|^2 &\leq \frac{1}{2} \| x_p - \bar{x} \|^2 + \frac{p}{2} \| b - A(x_p - \bar{x}) \|^2 = f_p (x_p, y_p). \label{BindingProbIneq}
%%%\end{align}
Now suppose that $\bm{u}^* \neq \bm{w}_p$. By (\ref{BindingProbIneq}) and the strict convexity of $f$ we have that $f(\bm{w}) < f(\bm{w}_p)$, for all $\bm{w} \in [\bm{w}_p, \bm{u}^*]$, that is for all $\bm{w}$ on the line segment connecting $\bm{w}_p$ and $\bm{u}^*$. By our choice of $\mathcal{B}$, $(\bm{w}_p)_i > 0$ for all $i \in \mathcal{I}$ and $(\bm{w}_p)_i = 0 = \bm{u}^*_i$ for all $i \in \mathcal{B}$. Hence, $\bm{w} \geq 0$ for all $\bm{w} \in [\bm{w}_p, \bm{u}^*]$ sufficiently close to $\bm{w}_p$, a contradiction. 
% NOTES ON CONTRADICTION: The contradiction follows since the points in $[x_p, u^*]$ are of the form $t x_p + (1 - t) u^*$, for $t \in [0, 1]$. Hence, for some $i \in \mathcal{I}$ and some range of values of $t \in [0,1]$, we have that $t (x_p)_i + (1 - t) (u^*)_i < 0$, which is a contradiction.
Thus, $\bm{u}^* = \bm{w}_p$.

Now observe that we can reformulate problem (\ref{BindingProb}) as the unconstrained optimization problem
\begin{align}
\min_{\bm{w}_+ \in \mathbb{R}^{|\mathcal{I}|}} \ \ \ \ g(\bm{w}_+) := \frac{1}{2} \| \bm{w}_+ - \bar{\bm{w}}_+ \|^2 + \frac{p}{2} \| \bm{c} + \bm{M} \bar{\bm{w}} - \bm{M}_+ \bm{w}_+ \|^2 \label{BindingProb1}
\end{align}
where $\bar{\bm{w}}_+ \in \mathbb{R}^{|\mathcal{I}|}$ is the vector with components $\bar{w}_i$, for all $i \in \mathcal{I}$, and $\bm{M}_+$ is the matrix with columns $M_i$ of $\bm{M}$, for all $i \in \mathcal{I}$. Since problem (\ref{BindingProb1}) is an unconstrained optimization problem with a strongly convex objective function, the first order optimality conditions are given by
\begin{align}
\bm{w}_+ \ \text{minimizes} \ g(\bm{w}) 
\ \ \text{if and only if} \ \ \nabla g (\bm{w}_+) = \frac{1}{p} \left( \bm{w}_+ - \bar{\bm{w}}_+ \right) - \bm{M}_+^{\intercal} \left( \bm{c} + \bm{M} \bar{\bm{w}} - \bm{M}_+ \bm{w}_+ \right) = 0. \label{gradientEq}
\end{align}
Letting $\bm{Q} \bm{\Sigma} \bm{P}^{\intercal}$ be a singular value decomposition of $\bm{M}_+$ and performing the change of variables $\bm{w}_+ = \bm{P} \bm{z}$ we obtain
%\begin{align}
$\nabla g (\bm{w}_+) 
%&= \frac{1}{p} \left( \bm{P} \bm{z} - \bar{\bm{w}}_+ \right) - \bm{P} \bm{\Sigma}^{\intercal} \bm{Q}^{\intercal} \left( \bm{c} + \bm{M} \bar{\bm{w}} - \bm{Q} \bm{\Sigma} \bm{P}^{\intercal} \bm{P} \bm{z} \right) \\
= \bm{P} \left( \frac{1}{p} \bm{z} + \bm{\Sigma}^{\intercal} \bm{\Sigma} \bm{z} \right) - \bm{P} \bm{\Sigma}^{\intercal} \bm{Q}^{\intercal} \left( \bm{c} + \bm{M} \bar{\bm{w}} \right) - \frac{1}{p} \bar{\bm{w}}_+$.
%\end{align}
Hence, multiplying both sides of equation (\ref{gradientEq}) by $\bm{P}^{\intercal}$ and simplifying yields
\begin{align}
\left( \frac{1}{p} \bm{I} + \bm{\Sigma}^{\intercal} \bm{\Sigma} \right) \bm{z} 
= \bm{\Sigma}^{\intercal} \bm{Q}^{\intercal} \left( \bm{c} + \bm{M} \bar{\bm{w}} \right) + \frac{1}{p} \bm{P}^{\intercal} \bar{\bm{w}}_+. \label{gradientEq1}
\end{align}
Now let $r = rank(\bm{M}_+)$ and let $\sigma_1 \geq \sigma_2 \geq \ldots \geq \sigma_r$ denote the non-zero singular values of $\bm{M}_+$. From equation (\ref{gradientEq1}), for $1 \leq i \leq r$, we have 
%\begin{align}
$\left( \frac{1}{p} + \sigma_i^2 \right) z_i = \sigma_i Q_i^{\intercal} (\bm{c} + \bm{M} \bar{\bm{w}}) + \frac{1}{p} \left( \bm{P}^{\intercal} \bar{\bm{w}}_+ \right)_i$,
%\end{align}
or equivalently, 
\begin{align}
z_i 
&= \frac{\sigma_i}{\sigma_i^2 + 1/p} Q_i^{\intercal} (\bm{c} + \bm{M} \bar{\bm{w}}) + \frac{1/p}{\sigma_i^2 + 1/p} \left( \bm{P}^{\intercal} \bar{\bm{w}}_+ \right)_i
= \frac{1}{\sigma_i} Q_i^{\intercal} (\bm{c} + \bm{M} \bar{\bm{w}}) + O \left( \frac{1}{p} \right). \label{zSol}
\end{align}
For $i > r$ it also follows from  equation (\ref{gradientEq1}) that
%\begin{align}
$z_i = \left( \bm{P}^{\intercal} \bar{\bm{w}}_+ \right)_i$. %\label{zSol1}
%\end{align} 
%From (\ref{zSol}) and (\ref{zSol1}) we have 
Combining these observations with (\ref{zSol}) yields that there exists a vector $\bm{u}^{\mathcal{B}} \in \mathbb{R}^n$ such that $\bm{z} = \bm{u}^{\mathcal{B}} + O \left( 1/p \right)$. As $\bm{w}_+ = \bm{P} \bm{z}$, we have that $\bm{w}_+ = \bm{P} \bm{w}^{\mathcal{B}} + O \left( 1/p \right)$. Since $\bm{w}_+$ is comprised of the positive components of $\bm{w}_p$, we conclude that 
\begin{align}
\bm{w}_p = \bm{v}^{\mathcal{B}} + O \left( 1/p \right) \label{xpFormula}
\end{align}
where $\bm{v}^{\mathcal{B}} \in \mathbb{R}^n$ is the vector satisfying $\left( \bm{v}^{\mathcal{B}} \right)_{\mathcal{B}} = 0$ and $\left( \bm{v}^{\mathcal{B}} \right)_{\mathcal{I}} = \bm{P} \bm{u}^{\mathcal{B}}$.
%%%is defined by 
%%%\begin{displaymath}
%%%\left( v_{\mathcal{B}} \right)_i = \left\{
%%%\begin{array}{lr}
%%%\left( w_{\mathcal{B}} \right)_i &: i \in \mathcal{B} \\
%%%0 &: i \in \mathcal{I}
%%%\end{array}
%%%\right.
%%%\end{displaymath}
From equation (\ref{xpFormula}) there exists a constant $c_{\mathcal{B}} \geq 0$ such that 
%\begin{align}
$\left\| \bm{w}_p - \bm{v}^{\mathcal{B}} \right\| 
\leq \frac{c_{\mathcal{B}}}{p}$. 
%\label{linearRate}
%\end{align}
Moreover, since the number of subsets of $\{ 1, 2, \ldots, n \}$ is finite, taking
%\begin{align}
$c := \max_{\mathcal{D} \subseteq \{ 1, \ldots, n \}} c_{\mathcal{D}}$
%\end{align}
yields
%\begin{align}
$\left\| \bm{w}_p - \bm{v}^{\mathcal{B}} \right\| \leq \frac{c}{p}$. 
%\label{linearRate1}
%\end{align}
As $p \geq 1$ was taken to be arbitrary, we conclude that this inequality 
%(\ref{linearRate1}) 
holds for all $p \geq 1$. 

Finally, we claim that $\bm{v}^{\mathcal{B}} = \bm{w}_{\infty}$, for $p$ sufficiently large. By way of contradiction, suppose that there exists a $\delta > 0$ such that $\left\| \bm{v}^{\mathcal{B}} - \bm{w}_{\infty} \right\| \geq \delta$, for all $p$ sufficiently large. Since $\bm{w}_p \to \bm{w}_{\infty}$ as $p \to \infty$ and $\left\| \bm{w}_p - \bm{v}^{\mathcal{B}} \right\| \leq \frac{c}{p}$ there exists a $N > 0$ such that $\| \bm{w}_p - \bm{w}_{\infty} \| < \frac{1}{2} \delta$ and $\left\| \bm{w}_p - \bm{v}^{\mathcal{B}} \right\| < \frac{1}{2} \delta$, for all $p \geq N$. Hence, for all $p \geq N$, it follows that
%\begin{align}
$\left\| \bm{v}^{\mathcal{B}} - \bm{w}_{\infty} \right\| 
%= \left\| \bm{v}^{\mathcal{B}} - \bm{w}_p + \bm{w}_p - \bm{w}_{\infty} \right\|
\leq \left\| \bm{v}^{\mathcal{B}} - \bm{w}_p \right\| +  \| \bm{w}_p - \bm{w}_{\infty} \|
< \delta$,
%\end{align}
a contradiction. Since for subsets $\mathcal{B} \subseteq \{ 1, 2, \ldots, n \}$ there are only finitely many such vectors $\bm{v}^{\mathcal{B}}$, we have that $\bm{v}^{\mathcal{B}} = \bm{w}_{\infty}$ for all $p$ sufficiently large. Therefore,
%\begin{align}
$\| \bm{w}_p - \bm{w}_{\infty} \| \leq \frac{c}{p}$
%\end{align}
for sufficiently large $p$.
\end{proof}

\iffalse
As a final comment for this section, note that the proof of Theorem~\ref{thm:feasibility-detection} provides some information on the size of the perturbation $\| \bm{y}_p \|$. In particular, from inequality (\ref{yinfBound.0}) there exists a constant $c$ such that $\| \bm{y}_p \| \leq \sqrt{\varepsilon + c^2/p}$. Hence, if the constraint set is feasible without perturbation then we have that $\varepsilon = 0$ which yields that $\| \bm{y}_p \| \leq \frac{c}{\sqrt{p}}$. We now turn our attention to establishing convergence results for each phase of NPASA.
\fi

%%%%%%%%%%%%%%%%%%%%%%%%%%%%%%%%%%%%%%%%%%%%%%%%%%%%%%%%%%%%%%%%%%%%%%%%%%%%%%%
%%% ----------------------- CONSTRAINT STEP ANALYSIS ---------------------- %%%
%%%%%%%%%%%%%%%%%%%%%%%%%%%%%%%%%%%%%%%%%%%%%%%%%%%%%%%%%%%%%%%%%%%%%%%%%%%%%%%
\subsection{Convergence Analysis for Constraint Step} \label{subsec:constraint-step-analysis}
In this section, we provide a convergence result for the constraint step in the LS algorithm. In particular, we establish a linear convergence rate for the constraint step scheme. Recall the iterative scheme for determining $\bm{w} \in \Omega$ such that $\bm{h} (\bm{w}) = \bm{0}$ used in the constraint step of the LS algorithm:
\begin{align}
(\bm{w}_{i+1}, \bm{y}_{i+1}) &= \arg\min \left\{\| \bm{w} - \bm{w}_i \|^2 + \| \bm{y} \|^2 : \nabla \bm{h} (\bm{w}_i) (\bm{w} - \bm{w}_i) + \frac{1}{\sqrt{p_i}} \bm{y} = - \bm{h} (\bm{w}_i), \bm{w} \in \Omega \right\} \label{NewtonYScheme} \\
\bm{w}_{i+1} &\gets \bm{w}_i + s_i ( \bm{w}_{i+1} - \bm{w}_i) \label{NewtonYScheme.0}
\end{align}
where $p_i \geq 1$ is a penalty parameter and where $s_i \leq 1$ is chosen such that
\begin{align}
\| \bm{h} (\bm{w}_i + s_i (\bm{w}_{i+1} - \bm{w}_{i})) \| \leq (1 - \tau \alpha_i s_i) \| \bm{h} (\bm{w}_i) \|. \label{NewtonArmijo}
\end{align} 
Here, $\alpha_i$ is defined by $\alpha_i := 1 - \| \bm{y}_i \|$ and to continue performing the scheme we require that $\alpha_i \in [\alpha, 1]$ for some fixed parameter $\alpha > 0$. Note that this is the iterative scheme in the constraint step of the LS algorithm, Algorithm~\ref{alg:cms}, where we have made the change of variables $\bm{y} \gets \frac{1}{\sqrt{p_i}} \bm{y}$ to simplify the analysis in Lemma~\ref{lem:Lem6.1}. For problems in (\ref{NewtonYScheme}) where the constraint set is nonempty for $\bm{y} = \bm{0}$, sufficiently large choices of $p$ will result in $\| \bm{y} \| < 1 - \alpha$, or equivalently $\alpha_i \geq \alpha$. For problems in (\ref{NewtonYScheme}) where the constraint set is empty for $\bm{y} = \bm{0}$ we may not be able to satisfy the requirement $\alpha_i > 0$, let alone $\alpha_i \geq \alpha$. Recall that in such a case NPASA branches to phase one and performs a global step. 
%%%% -------- Not necessary for the following analysis -------- %%%%
\iffalse
Since this scheme involves the problem considered in Proposition \ref{PpProp}, we define 
\begin{align}
\bm{w}_{i+1, \infty} = \arg \min_{\bm{w} \in \mathcal{S}_i} \| \bm{w} - \bm{w}_i \|^2 
\end{align}
where $\mathcal{S}_i = \arg \min \left\{ \| \bm{h} (\bm{w}_i) + \nabla \bm{h} (\bm{w}_i) (\bm{w} - \bm{w}_i) \|^2, \bm{w} \geq 0 \right\}$. Additionally, we define 
\begin{align}
\bm{y}_{i+1, \infty} = \bm{h} (\bm{w}_i) + \nabla \bm{h} (\bm{w}_i) (\bm{w}_{i+1, \infty} - \bm{w}_i).
\end{align}
\fi
%%%% ---------------------------------------------------------- %%%%
An approach similar to (\ref{NewtonYScheme}) -- (\ref{NewtonYScheme.0}) was previously considered 
%in the literature 
where iterates are given by 
\begin{align}
\bm{w}_{i+1} &= \arg\min \left\{\| \bm{w} - \bm{w}_i \|^2 : \nabla \bm{h} (\bm{w}_i) (\bm{w} - \bm{w}_i) = - \bm{h} (\bm{w}_i), \bm{w} \in \Omega \right\} \label{NewtonScheme} \\
\bm{w}_{i+1} &\gets \bm{w}_i + s_i (\bm{w}_{i+1} - \bm{w}_{i}) \label{NewtonScheme.0}
\end{align}
and $s_i \leq 1$ is chosen using (\ref{NewtonArmijo}). Motivation and discussion of the scheme in (\ref{NewtonScheme}) -- (\ref{NewtonScheme.0}) can be found in \cite{Daniel1973, Hager1993, Robinson1972}. In particular, in Section~3 of \cite{Daniel1973} Daniel considered the solvability of the constraint $\nabla \bm{h} (\bm{w}_i) (\bm{w} - \bm{w}_i) = - \bm{h} (\bm{w}_i)$ found in (\ref{NewtonScheme}). By introducing the slack variable, $\bm{y}$, into our update scheme, it is no longer necessary to consider the solvability of this problem as the constraint set is always nonempty. However, %the convergence analysis performed in Theorem 2.1 of \cite{Hager1993} 
it is necessary to perform updated convergence analysis for scheme (\ref{NewtonScheme}) -- (\ref{NewtonScheme.0}) to account for the presence of the slack variable $\bm{y}$ and penalty parameter $p_i$ in (\ref{NewtonYScheme}). 

%With global convergence established, we consider some additional lemmas illustating that phase two may not branch back to phase one even when when the current iterate is not located in a neighborhood of a stationary point. 

Before presenting this result, we consider some heuristics for the constraint step that should encourage global convergence of NPASA. One of the updates in this implementation of the primal-dual method from \cite{Hager1993} is the use of the Newton step with feasibility detection in the constraint step of the algorithm. Previously, the convergence analysis required the assumption that the linearized constraints in (\ref{NewtonScheme}) had a solution at each iterate $i$. While we no longer require this assumption, we do require some method for deciding when a perturbation $\| \bm{y}_{i+1} \|$ in (\ref{NewtonYScheme}) is small enough for us to continue with the constraint step iteration. From our analysis in Subection~\ref{subsec:infeas}, in particular inequality (\ref{objBound1}), we have that if the linearized constraints are feasible then $\| \bm{y}_{i+1} \| = O \left( \frac{1}{\sqrt{p_i}} \right)$. With this in mind, suppose now that we fix two parameters $\alpha \in (0, 1]$ and $\beta \geq 1$ such that $\frac{1}{\beta} < 1 - \alpha$. Then we choose our penalty parameter $p_i = \max \left\{ \beta^2, \| \bm{h} (\bm{w}_i) \|^{-2} \right\}$ and consider the linearized constraints to be feasible when $\| \bm{y}_{i+1} \| \leq 1 - \alpha$. As $p_i \geq \beta^2$, if the linearized constraint set is feasible without perturbation, then under the updated scheme in (\ref{NewtonYScheme}) we have that $\| \bm{y}_{i+1} \| = O \left( \frac{1}{\sqrt{p_i}} \right) = O \left( \frac{1}{\beta} \right)$. By our choice of $\alpha$ and $\beta$, we now have a good chance of satisfying the criterion $\| \bm{y}_{i+1} \| \leq 1 - \alpha$ when the linearized constraint is feasible without perturbation. If it is infeasible by a perturbation $\gamma$ in norm, then from inequality (\ref{objBound1}) we have that $\| \bm{y}_{i+1} \|^2 \leq \frac{c}{p_i} + \gamma$. Hence, for small perturbations we still may be able to accept iterates and move towards satisfying the constraint step. %As an example, one suitable choice for these parameters might be $\alpha = 0.5$ and $\beta = 10$. A more agressive approach might take $\alpha = 0.5$ and $\beta = 100$. Note that to enforce smaller perturbation sizes $\alpha$ should be taken closer to one.

%The following result is similar to Lemma 6.1 in \cite{Hager1993} but accounts for the modifications to the constraint step and our heuristics for choosing a penalty parameter $p_i$. %This result establishes a global linear convergence rate for the constraint step in the CMS algorithm. %that requires far fewer assumptions than the local quadratic convergence result established in Section~\ref{section:npasa-phase-two}. 

\begin{lemma} \label{lem:Lem6.1}
%Let $\alpha \in (0,1]$ and $\beta \geq 1$ be given such that $\frac{1}{\beta} < 1 - \alpha$. 
Suppose that there exists a constant $M$ such that $\| \bm{w}_{i+1} - \bm{w}_i \| \leq M$ for every $i$ %, that there exists an $\alpha \in (0,1]$ such that $\| \bm{y}_{i+1} \| \leq 1 - \alpha$ for every $i$, 
and there exists a constant $r > 0$ such that $\nabla \bm{h}$ is Lipschitz continuous with modulus $\kappa$ in the ball $B (\bm{w}_i, r)$, for each $i$. Then for each $i$ where $\bm{h} (\bm{w}_i) \neq \bm{0}$ and $\| \bm{y}_i \| \leq 1 - \alpha$, the iterate $\bm{w}_{i+1}$ defined in (\ref{NewtonYScheme}) -- (\ref{NewtonYScheme.0}) exists and satisfies
\begin{align}
\| \bm{h} (\bm{w}_{i+1}) \| \leq (1 - \gamma_i \alpha_i \tau) \| \bm{h} (\bm{w}_i) \|,
\end{align}
where $\alpha_i = 1 - \| \bm{y}_i \|$ and 
\begin{align}
\gamma_i = \min \left\{ 1, \frac{r \sigma}{M}, \frac{2 \alpha \sigma (1 - \tau) \| \bm{h} (\bm{w}_i) \|}{\kappa M^2} \right\}.
\end{align}
Thus, either $\bm{h} (\bm{w}_i) = \bm{0}$ after a finite number of iterations or $\displaystyle \lim_{i \to \infty} \bm{h} (\bm{w}_i) = \bm{0}$.
\end{lemma}
\begin{proof}
First note that by the definition of $\alpha_i$ and our hypothesis that $\| \bm{y}_i \| \leq 1 - \alpha$ it follows that $\alpha \leq \alpha_i$. Now let $\bm{w}, \bm{d} \in \mathbb{R}^n$ and $s \in \mathbb{R}$ such that $\bm{h}$ is $r$-Lipschitz continuous on the line segment $[\bm{w}, \bm{w} + s \bm{d}]$. By the fundamental theorem of calculus it follows that
\begin{align}
\bm{h} (\bm{w} + s \bm{d}) 
%&= \bm{h} (\bm{w}) + \int_0^s \nabla \bm{h} (\bm{w} + t \bm{d}) \bm{d} \ dt \\
&= \bm{h} (\bm{w}) + s \nabla \bm{h} (\bm{w}) \bm{d} + \int_0^s \left( \nabla \bm{h} (\bm{w} + t \bm{d}) - \nabla \bm{h} (\bm{w}) \right) \bm{d} \ dt. \label{eq:6.1.1}
\end{align}
By (\ref{NewtonYScheme}), we have that
\begin{align}
\nabla \bm{h} (\bm{w}_i) (\bm{w}_{i+1} - \bm{w}_i) = - \bm{h} (\bm{w}_i) - \frac{1}{\sqrt{p}} \bm{y}_i. \label{eq:6.1.2}
\end{align}
Hence, setting $\bm{w} = \bm{w}_i$, $\bm{y} = \bm{y}_i$, and $\bm{d} = \bm{w}_{i+1} - \bm{w}_i$  in (\ref{eq:6.1.1}) and combining the resulting equation with (\ref{eq:6.1.2}) we have that
\begin{align}
\| \bm{h} (\bm{w}_i + s (\bm{w}_{i+1} - \bm{w}_i)) \| 
&\leq \| \bm{h} (\bm{w}_i) + s \nabla \bm{h} (\bm{w}_i) (\bm{w}_{i+1} - \bm{w}_i) \| + \frac{\kappa s^2 M^2}{2} \\
&\leq \| \bm{h} (\bm{w}_i) - s \bm{h} (\bm{w}_i) \| +  \frac{s}{\sqrt{p}} \| \bm{y}_i \| + \frac{\kappa s^2 M^2}{2} \\
&\leq (1 - s) \| \bm{h} (\bm{w}_i) \| + s \| \bm{y}_i \| \| \bm{h}(\bm{w}_i) \| + \frac{\kappa s^2 M^2}{2} \\
&= \left(1 - s \left( 1 - \| \bm{y}_i \| \right) \right) \| \bm{h} (\bm{w}_i) \| + \frac{\kappa s^2 M^2}{2} \\
&= \left(1 - \alpha_i s \right) \| \bm{h} (\bm{w}_i) \| + \frac{\kappa s^2 M^2}{2}. \label{eq:6.1.2.1}
\end{align}
Now assume that $\bm{h} (\bm{w}_i) \neq \bm{0}$. By our choice of $s_i$ in NPASA, we have that either
%\begin{align}
$s_i = 1$
\label{eq:6.1.3}
%\end{align}
or
\begin{align}
(1 - \tau \alpha_i \sigma^{-1} s_i) \| \bm{h} (\bm{w}_i) \| \leq \| \bm{h} (\bm{w}_i - \sigma^{-1} s_i (\bm{w}_{i+1} - \bm{w}_i)) \|. \label{eq:6.1.4}
\end{align}
%If (\ref{eq:6.1.3}) is satisfied then we know the value of $s_i$.
If $s_i = 1$ then there is nothing to be done. Hence, we suppose that (\ref{eq:6.1.4}) is satisfied and consider two subcases.\\
\noindent \emph{Case} ($i$): $\sigma^{-1} s_i \| \bm{w}_{i+1} - \bm{w}_i \| > r$. By our hypothesis that $\| \bm{w}_{i+1} - \bm{w}_i \| \leq M$ we have the lower bound
\begin{align}
s_i > \frac{r \sigma}{\| \bm{w}_{i+1} - \bm{w}_i \|} \geq \frac{r \sigma}{M}. \label{eq:6.1.5}
\end{align}

\noindent \emph{Case} ($ii$): $\sigma^{-1} s_i \| \bm{w}_{i+1} - \bm{w}_i \| \leq r$. 
First, note that by setting $s = \sigma^{-1} s_i$ in (\ref{eq:6.1.2.1}) we have that
\begin{align}
\| \bm{h} (\bm{w}_i + \sigma^{-1} s_i (\bm{w}_{i+1} - \bm{w}_i)) \| 
&\leq \left(1 - \alpha_i \sigma^{-1} s_i \right) \| \bm{h} (\bm{w}_i) \| + \frac{\kappa s_i^2 M^2}{2 \sigma^2}. \label{eq:6.1.6}
\end{align}
Combining (\ref{eq:6.1.4}) with (\ref{eq:6.1.6}) yields
%\begin{align}
$(1 - \tau \alpha_i \sigma^{-1} s_i) \| \bm{h} (\bm{w}_i) \|
\leq \left(1 - \alpha_i \sigma^{-1} s_i \right) \| \bm{h} (\bm{w}_i) \| + \frac{\kappa s_i^2 M^2}{2 \sigma^2}$,
%\end{align}
which, when simplified, gives us
\begin{align}
\frac{2 \alpha_i \sigma (1 - \tau) \| \bm{h} (\bm{w}_i) \|}{\kappa M^2}
&\leq s_i. \label{eq:6.1.7}
\end{align}
Lastly, note that by the definition of $\alpha_i$ and our hypothesis that $\| \bm{y}_i \| \leq 1 - \alpha$ it follows that $\alpha \leq \alpha_i$. Thus, by (\ref{eq:6.1.3}), (\ref{eq:6.1.5}), and (\ref{eq:6.1.7}) we conclude that the choice of $\gamma_i$ stated in the lemma satisfies the Armijo line search at iterate $i$.
\end{proof}

%%%%%%%%%%%%%%%%%%%%%%%%%%%%%%%%%%%%%%%%%%%%%%%%%%%%%%%%%%%%%%%%%%%%%%%%%%%%%%%
%%% ----------------------- MULTIPLIER STEP ANALYSIS ---------------------- %%%
%%%%%%%%%%%%%%%%%%%%%%%%%%%%%%%%%%%%%%%%%%%%%%%%%%%%%%%%%%%%%%%%%%%%%%%%%%%%%%%
\subsection{Convergence Analysis for Multiplier Step} \label{subsec:multiplier-step-analysis} 
We now provide a convergence result for the multiplier step in the LS algorithm. We note that this result is similar to Lemma 6.2 in \cite{Hager1993} and we omit the proof as it follows from the same argument provided in \cite{Hager1993} when the function $\bm{K}(\bm{\Lambda}, \bm{x})$ is replaced with $E_{m,1} (\bm{z}, \bm{\nu}, \bm{\eta})$.

\begin{lemma} \label{lem:mult-step}
Suppose that $f, \bm{h} \in \mathcal{C}^1 (\Omega)$ and that $\Omega$ is compact. Then there exists a sequence of natural numbers $\{ n_i \}_{i = 0}^{\infty}$ such that $\{ \bm{z}_{n_i} \}_{i = 0}^{\infty}$ is a convergent subsequence with limit point $\bm{z}^*$. Furthermore, if (\textbf{LICQ}) holds at $\bm{z}^*$ then there exists a vector $\bm{\nu}^*$ such that
%\begin{align}
$\min_{\bm{\eta} \geq \bm{0}} \ E_{m,1} (\bm{z}^*, \bm{\nu}^*, \bm{\eta}) = 0$.
%\end{align}
\end{lemma}

We note that Lemma~\ref{lem:mult-step} could be used to modify the multiplier step of the LS Algorithm, Algorithm~\ref{alg:cms}, so that it is designed to identify a convergent subsequence instead of a convergent sequence. However, as phase two of NPASA is developed for fast local convergence we have intentionally designed the multiplier step to continue only when the multiplier error, $E_{m,1}$, is decreasing at a linear rate or higher. This concludes our convergence analysis for phase two of NPASA.

\section{Convergence Analysis of NPASA} \label{section:npasa-global-convergence}
At this point, we have established convergence results for each problem solved in NPASA and we are ready to state and prove a global convergence result for NPASA.
\begin{theorem}[NPASA Global Convergence Theorem] \label{thm:npasa-global-conv}
Suppose that $\Omega$ is compact and that $f, \bm{h} \in \mathcal{C}^1 (\Omega)$. Suppose that NPASA (Algorithm~\ref{alg:npasa}) with $\varepsilon = 0$ generates a sequence $\{ (\bm{x}_k, \bm{\lambda}_k, \bm{\mu}_k) \}_{k=0}^{\infty}$ with $\bm{x}_k \in \Omega$ and let $\mathcal{S}_j$ be the set of indices such that if $k \in \mathcal{S}_j$ then $\bm{x}_k$ is generated in phase $j$ of NPASA, for $j \in \{ 1, 2 \}$. Suppose that the following assumptions hold:
\begin{enumerate}
\item[\emph{(H1)}] For every $k \in \mathcal{S}_1$, assumptions (G1) and (G2) in Section~\ref{section:npasa-phase-one} hold at $(\bm{x}_{k-1}, \bm{\nu}_{k-1}, q_k)$, where $\bm{\nu}_{k-1} = Proj_{[-\bar{\lambda},\bar{\lambda}]} (\bm{\lambda}_{k-1})$ and $\bar{\lambda} > 0$ is a scalar parameter.
\item[\emph{(H2)}] If $\bm{x}^*$ is a subsequential limit point of $\{ \bm{x}_k \}_{k=0}^{\infty}$ then (\textbf{LICQ}) holds at $\bm{x}^*$.
%\item[(H3)] For $k$ sufficiently large, the penalty parameter in phase one is updated using the formula $\displaystyle q_{k+1} = \max \left\{ \phi, (e_k)^{-1} \right\} q_k$, where $e_k := \min \{ E_1 (\bm{x}_k, \bm{\lambda}_k, \bm{\mu}_k), e_{k-1} \}$ and $e_0 := E_1 (\bm{x}_0, \bm{\lambda}_0, \bm{\mu}_0)$.
\end{enumerate}
Then 
\begin{align}
\liminf_{k \to \infty} E_1 (\bm{x}_k, \bm{\lambda}_k, \bm{\mu}_k) = 0. \label{result:npasa-global-conv}
\end{align}
\end{theorem}

\begin{proof}
To prove this result we consider three different scenarios. Before considering each scenario, we claim that the hypotheses of Theorem~\ref{thm:pasa-aug-lag-global-conv} are satisfied across all cases. To see this, from assumption (H1) it follows that Theorem~\ref{thm:pasa-aug-lag-mod-global-conv} holds. By Theorem~\ref{thm:pasa-aug-lag-mod-global-conv}, Algorithm~\ref{alg:gs} can always generate a point satisfying hypothesis (1) in Theorem~\ref{thm:pasa-aug-lag-global-conv}. Additionally, (2) and (3) in Theorem~\ref{thm:pasa-aug-lag-global-conv} are satisfied by the definition of Algorithm~\ref{alg:npasa}. % even with the modified penalty parameter formula in (H3). 
The final assumption in Theorem~\ref{thm:pasa-aug-lag-global-conv} follows from (H2). Hence, all of the hypotheses of Theorem~\ref{thm:pasa-aug-lag-global-conv} are satisfied. We now consider the three cases. \\
{\bfseries Case 1: $\mathcal{S}_1$ is finite.} In this case, there exists an nonnegative integer $N$ such that $k \geq N$ implies $k \in \mathcal{S}_2$. As the branching criterion for phase two to phase one in line 42 of Algorithm~\ref{alg:npasa} are never satisfied for $k > N$, we have that
%\begin{align}
$E_1 (\bm{x}_k, \bm{\lambda}_k, \bm{\mu}_k) \leq \theta E_1 (\bm{x}_{k-1}, \bm{\lambda}_{k-1}, \bm{\mu}_{k-1})$,
%\end{align}
for all $k > N$. In particular, we have 
%\begin{align}
$E_1 (\bm{x}_{N+k}, \bm{\lambda}_{N+k}, \bm{\mu}_{N+k}) \leq \theta^k E_1 (\bm{x}_{N}, \bm{\lambda}_{N}, \bm{\mu}_{N})$,
%\end{align}
for all $k \geq 0$. As $\theta \in (0, 1)$, it follows that
%\begin{align}
$\lim_{k \to \infty} E_1 (\bm{x}_k, \bm{\lambda}_k, \bm{\mu}_k) = 0$. %\label{eq:npasa-global-conv.1}
%\end{align}
\\
{\bfseries Case 2: $\mathcal{S}_2$ is finite.} In this case, there exists an nonnegative integer $N$ such that $k \geq N$ implies $k \in \mathcal{S}_1$. Now suppose phase one of NPASA is given the starting point $(\bm{x}_N, \bm{\lambda}_N, \bm{\mu} (\bm{x}_N,1))$. Then phase one of NPASA generates all remaining iterates and, as the hypotheses of Theorem~\ref{thm:pasa-aug-lag-global-conv} are satisfied, it follows that $E_1 (\bm{x}^*, \bm{\lambda}^*, \bm{\mu}(\bm{x}^*,1)) = 0$. 
%%%%%%%%%%%%%%%%%%%%%%%%%% (No longer necessary)
\iffalse
Recalling that $E_1 (\bm{x}, \bm{\lambda}, \bm{\mu}) \leq E_0 (\bm{x}, \bm{\lambda}, \bm{\mu})$ for all $\bm{x} \in \Omega$ and $\bm{\mu} \geq \bm{0}$, we have that 
\begin{align}
E_1 (\bm{x}_k, \bm{\lambda}_k, \bm{\mu}(\bm{x}_k,1)) 
\leq E_0 (\bm{x}_k, \bm{\lambda}_k, \bm{\mu}(\bm{x}_k,1)), \label{eq:npasa-global-conv.1.1}
\end{align}
for all $k \geq N$. As $\bm{x}^*$ is a limit point of a subsequence of $\{ \bm{x}_k \}_{k = 0}^{\infty}$, it follows that 
\begin{align}
\liminf_{k \to \infty} E_1 (\bm{x}_k, \bm{\lambda}_k, \bm{\mu}(\bm{x}_k,1)) 
&\leq \liminf_{k \to \infty} E_0 (\bm{x}_k, \bm{\lambda}_k, \bm{\mu}(\bm{x}_k,1))
= E_0 (\bm{x}^*, \bm{\lambda}^*, \bm{\mu}(\bm{x}^*,1)) 
= 0 \label{eq:npasa-global-conv.1.2}
\end{align}
\fi
%%%%%%%%%%%%%%%%%%%%%%%%%%%%%%%%%%%%%%%%%%%%%%%%
As $\bm{x}^*$ is a limit point of a subsequence of $\{ \bm{x}_k \}_{k = 0}^{\infty}$, it follows that 
%\begin{align}
$\liminf_{k \to \infty} E_1 (\bm{x}_k, \bm{\lambda}_k, \bm{\mu}(\bm{x}_k,1)) 
= E_1 (\bm{x}^*, \bm{\lambda}^*, \bm{\mu}(\bm{x}^*,1)) 
= 0$, 
%\label{eq:npasa-global-conv.1.2}
%\end{align}
which proves (\ref{result:npasa-global-conv}). \\
{\bfseries Case 3: $\mathcal{S}_1$ and $\mathcal{S}_2$ are infinite.} %For simplicity, let $\mathcal{S}_1 = \{ n_1, n_2, \ldots \}$ and $\mathcal{S}_2 = \{ m_1, m_2, \ldots \}$. 
Suppose that $k \in \mathcal{S}_1$. By (H1), phase one will always generate a point that satisfies the criterion for branching to phase two, that is
%\begin{align}
$E_{m,0} (\bm{x}_{k}, \bm{\lambda}_{k}, \bm{\mu}_{k}) 
\leq \theta E_c (\bm{x}_{k})$. %\label{eq:npasa-global-conv.2}
%\end{align}
Combining this with $E_{m,1} (\bm{x}_{k}, \bm{\lambda}_{k}, \bm{\mu}_{k}) \leq E_{m,0} (\bm{x}_{k}, \bm{\lambda}_{k}, \bm{\mu}_{k})$ from Lemma~\ref{lem:E0-E1-relationship} yields
%which combined with (\ref{eq:npasa-global-conv.2}) yields that 
\begin{align}
E_{m,1} (\bm{x}_{k}, \bm{\lambda}_{k}, \bm{\mu}_{k}) 
\leq \theta E_c (\bm{x}_{k}). \label{eq:npasa-global-conv.2.1}
\end{align}
Now let $\{ n_1, n_2, \ldots \} \subset \mathcal{S}_1$ be the indices such that $\bm{x}_{n_k}$ is a term in the subsequence converging to $\bm{x}^*$. From equation (\ref{eq:phase-one-convergence.3}) in the proof of Theorem~\ref{thm:pasa-aug-lag-global-conv}, there exists a constant $\Lambda$ such that 
\begin{align}
\| \bm{h} (\bm{x}_{n_k}) \| \leq \frac{\Lambda}{2 q_{n_k}}, \label{eq:npasa-global-conv.3}
\end{align}
for all $k$ sufficiently large. By the update formula for $q_k$ in line 4 of Algorithm~\ref{alg:npasa} and the fact that $\phi > 1$, we have that $q_k \to \infty$ as $k \to \infty$. Hence, for $k$ sufficiently large, $q_{n_k} \geq \Lambda$. In particular, by the update formula for $q_k$ we have that
\begin{align}
q_{n_k} \geq \frac{q_{n_{k-1}}}{e_{n_k - 1}} \geq \frac{\Lambda}{e_{n_k - 1}}, \label{eq:npasa-global-conv.4}
\end{align}
for all $k$ sufficiently large. As $E_c (\bm{x}) = \| \bm{h} (\bm{x}) \|^2$, it now follows from (\ref{eq:npasa-global-conv.2.1}), (\ref{eq:npasa-global-conv.3}), and (\ref{eq:npasa-global-conv.4}) that
\begin{align}
E_1 (\bm{x}_{n_k}, \bm{\lambda}_{n_k}, \bm{\mu}_{n_k}) 
%&= \sqrt{E_{m,1} (\bm{x}_{n_k}, \bm{\lambda}_{n_k}, \bm{\mu}_{n_k}) + E_c (\bm{x}_{n_k})} 
\leq \sqrt{(1 + \theta) E_c (\bm{x}_{n_k})}
%= \sqrt{1 + \theta} \ \| \bm{h} (\bm{x}_{n_k}) \|
\leq \sqrt{1 + \theta} \left( \frac{\Lambda}{2 q_{n_k}} \right)
\leq \left( \frac{\sqrt{1 + \theta}}{2} \right) e_{n_k - 1}. \label{eq:npasa-global-conv.5}
\end{align}
As $\displaystyle e_{n_k - 1} = \min_{0 \leq i < n_k } E_1 (\bm{x}_i, \bm{\lambda}_i, \bm{\mu}_i)$, it follows from (\ref{eq:npasa-global-conv.5}) that
\begin{align}
E_1 (\bm{x}_{n_k}, \bm{\lambda}_{n_k}, \bm{\mu}_{n_k})
\leq \left( \frac{\sqrt{1 + \theta}}{2} \right) E_1 (\bm{x}_{n_{k-1}}, \bm{\lambda}_{n_{k-1}}, \bm{\mu}_{n_{k-1}}), \label{eq:npasa-global-conv.5.1}
\end{align}
for all $k$ sufficiently large. 

On the other hand, suppose that $k \in \mathcal{S}_2$. By the branching criterion for phase two, we have that 
\begin{align}
E_1 (\bm{x}_{k}, \bm{\lambda}_{k}, \bm{\mu}_{k}) 
&\leq \theta E_1 (\bm{x}_{k - 1}, \bm{\lambda}_{k - 1}, \bm{\mu}_{k - 1}). \label{eq:npasa-global-conv.6}
\end{align}
As $\theta \in (0,1)$ by definition, we also have that $(\sqrt{1 + \theta})/2 \in (0,1)$. Thus, combining (\ref{eq:npasa-global-conv.5.1}) and (\ref{eq:npasa-global-conv.6}) yields (\ref{result:npasa-global-conv}).
\end{proof}

%Theorem~\ref{thm:npasa-global-conv} serves to illustrate a small set of assumptions under which global convergence of NPASA can be established.

\section{Conclusion} \label{section:conclusion}
In this paper, we presented a method for solving nonlinear programs, NPASA, and established global convergence properties for NPASA. In particular, under a small set of assumptions we established global convergence properties for each of the three subproblems solved during NPASA which, when combined, were able to ensure global convergence of NPASA. As noted in the introduction, a companion paper \cite{diffenderfer2020local} focuses on establishing local quadratic convergence for NPASA. In the future, we plan on implementing this approach and performing benchmarking tests to compare the NPASA algorithm to leading methods for solving nonlinear programs.

This work was performed, in part,  under the auspices of the U.S. Department of Energy by Lawrence Livermore National Laboratory under Contract DE-AC52-07NA27344.
%\input{acknowledgements}
% Add bibliography before appendix
\bibliographystyle{siamplain}
\bibliography{paper}
% Add appendix
%\clearpage
\appendix
\section{Constrained Optimization Definitions and Results} \label{appendixa}
In this appendix, we highlight key definitions and results from the constrained optimization theory that were used as assumptions in our analysis of NPASA. Statements of these results are included here for completeness but we note that more details on these theorems and proofs can be found by referencing \cite{NandW, bertsekas1995nonlinear}. As the hypotheses of many of these results are used as assumptions in establishing global and local convergence properties for NPASA, these results are only referenced in Section~\ref{subsec:notation} where we provide simplified abbreviations for the hypotheses of these theorems. In this appendix, given $i \in \mathbb{N}$ note that we will write $[i]$ to denote the set of integers $\{ 1, 2, \ldots, i \}$.

\iffalse
\begin{definition} \label{def:li}
A vector $\bm{x}$ is said to satisfy the {\bfseries linear independence condition} for problem (\ref{prob:main-nlp}) if the matrix {\footnotesize $\begin{bmatrix} \nabla \bm{h} (\bm{x}) \\ \bm{A} \end{bmatrix}$} is of full row rank.
\end{definition}
\fi

\begin{definition} \label{def:licq}
A vector $\bm{x}$ is said to satisfy the {\bfseries linear independence constraint qualification condition} for problem (\ref{prob:main-nlp}) if the matrix {\footnotesize $\displaystyle \begin{bmatrix} \nabla \bm{h} (\bm{x}) \\ \bm{A}_{\mathcal{A}(\bm{x})} \end{bmatrix}$} is of full row rank.
\end{definition}

%Note that the condition in Definition~\ref{def:li} is intentionally stronger than the classical condition in Definition~\ref{def:licq}. 
As in Section~\ref{subsec:notation}, note that %if $\bm{x}$ satisfies the linear independence condition then we abbreviate this by writing (\textbf{LI}) holds at $\bm{x}$ and 
if $\bm{x}$ satisfies the linear independence constraint qualification condition then we abbreviate this by writing \emph{(\textbf{LICQ}) holds at $\bm{x}$}. We now state the first order optimality or KKT conditions.

\begin{theorem}[Karush-Kuhn-Tucker Conditions \cite{NandW}] \label{thm:kkt}
Suppose $\bm{x}^*$ is a local solution of problem (\ref{prob:main-nlp}), that $f, \bm{h} \in \mathcal{C}^1$, and (\textbf{LICQ}) holds at $\bm{x}^*$. Then there exist KKT multipliers $\bm{\lambda}^*$ and $\bm{\mu}^*$ such that the following conditions hold at $(\bm{x}^*, \bm{\lambda}^*, \bm{\mu}^*)$:
\begin{enumerate}[leftmargin=2\parindent,align=left,labelwidth=\parindent,labelsep=7pt]
	\item[(\textbf{KKT.1})] Gradient of Lagrangian equals zero: $\nabla_x \mathcal{L} (\bm{x}^*, \bm{\lambda}^*, \bm{\mu}^*)^{\intercal} = \bm{0}$
	\item[(\textbf{KKT.2})] Satisfies equality constraints: $\bm{h} (\bm{x}^*) = \bm{0}$
	\item[(\textbf{KKT.3})] Satisfies inequality constraints: $\bm{r} (\bm{x}^*) \leq \bm{0}$
	\item[(\textbf{KKT.4})] Nonnegativity of inequality multipliers: $\bm{\mu}^* \geq \bm{0}$
	\item[(\textbf{KKT.5})] Complementary slackness: $r_i (\bm{x}^*) \mu_i^* = 0$ for $i \in [m]$.
\end{enumerate}
\end{theorem}

%Before stating second order optimality conditions we first provide several definitions required for the statements of these theorems.

\begin{definition} \label{def:scs}
A point $(\bm{x}, \bm{\lambda}, \bm{\mu})$ is said to satisfy {\bfseries strict complementary slackness} for problem (\ref{prob:main-nlp}) if it satisfies (\textbf{KKT.1}) -- (\textbf{KKT.5}) and exactly one of $r_i (\bm{x})$ and $\mu_i$ is zero for each $i \in [m]$.
\end{definition}

\clearpage

\end{document}